\documentclass[11pt]{article} 
\usepackage{amsmath, amsthm, amssymb, mathrsfs, enumerate} 
\usepackage{graphicx,amssymb,upref}
\usepackage[title,titletoc]{appendix}
\usepackage{thmtools, thm-restate}
\usepackage{url,hyperref} 

\usepackage[a4paper,margin=2cm, centering]{geometry}

\usepackage{ifpdf} 
\ifpdf
    \usepackage[all,pdf]{xy} 
\else
    \input xy
    \xyoption{all}
    \xyoption{2cell}
    \xyoption{v2}
\fi
\SelectTips{cm}{11} 

\newcommand{\cU}{\mathscr{U}} 
\newcommand{\cK}{\mathscr{K}} 
\newcommand{\cC}{\mathscr{C}}

\newcommand{\Dec}{\mathrm{Dec}}
\newcommand{\sk}{\mathrm{sk}}
\newcommand{\cosk}{\mathrm{cosk}}

\newcommand{\Set}{\mathit{Set}}
\newcommand{\sSet}{\mathit{sSet}}

\newcommand{\cKB}{\mathscr{K}_{/B}}

\newcommand{\Map}{\mathrm{Map}} 
\newcommand{\supp}{\mathrm{supp}} 
 
\newcommand{\Aut}{\mathrm{Aut}}

\newcommand{\ZZ}{\mathbb{Z}} 

\DeclareMathOperator*{\colim}{\mathrm{colim}}

\theoremstyle{plain}
\newtheorem{theorem}{Theorem}
\newtheorem{lemma}[theorem]{Lemma}
\newtheorem{proposition}[theorem]{Proposition}
\newtheorem{corollary}[theorem]{Corollary}

\theoremstyle{plain}

\theoremstyle{definition}
\newtheorem{definition}[theorem]{Definition}

\theoremstyle{remark}

\newcommand*\samethanks[1][\value{footnote}]{\footnotemark[#1]}

\title{Simplicial principal bundles in parametrized spaces} 
\author{
David M.\ Roberts\thanks{Supported by the Australian Research Council 
(grant number DP120100106)}\\
\href{mailto:david.roberts@adelaide.edu.au}{\texttt{david.roberts@adelaide.edu.au}}\\ 
School of Mathematical Sciences\\ 
The University of Adelaide\\ 
Adelaide SA 5005 \\ 
Australia  
\and 
Danny Stevenson\samethanks[1] \thanks{Supported by the 
 Engineering and Physical Sciences Research Council (grant number EP/I010610/1)} \\
 \href{mailto:daniel.stevenson@adelaide.edu.au}{\texttt{daniel.stevenson@adelaide.edu.au}}\\
School of Mathematical Sciences\\ 
The University of Adelaide\\ 
Adelaide SA 5005 \\ 
Australia 
}

\begin{document} 

\maketitle

\begin{abstract}

	In this paper we study the classifying 
	theory of principal bundles in the parametrized setting, 
	motivated by recent interest in higher gauge theory.  Using simplicial techniques, we construct a product-preserving 
	classifying space functor for groups in the category of spaces over a fixed space $B$.  
	Additionally, we prove that the fiberwise geometric realization functor sends a large class 
	of simplicial parametrized principal bundles to ordinary parametrized principal bundles.  
	 As an  application we show that the fiberwise geometric
	realization of the universal simplicial principal bundle for a simplicial
	group $G$ in the category of spaces over $B$ gives rise  to a parametrized 
	principal bundle with structure group $|G|$.


\end{abstract}

\tableofcontents

\section{Introduction}

The construction of a classifying space for a topological group is
conveniently done using simplicial techniques, namely via the geometric realization of a 
certain simplicial space.  Similarly a model for the
universal principal bundle can be constructed as the geometric realization
of a certain simplicial  principal bundle.  The utility of these constructions rests in part 
on the fact that the classifying space functor so obtained is product-preserving.  The aim of this paper is to extend these 
constructions to the parametrized setting of \cite{MaySig}, in which the category of 
topological spaces is replaced by a suitable category of spaces over a fixed 
space $B$.  

We recall the setting for parametrized homotopy theory from \cite{MaySig}. Let 
$\cK$ denote the category of $k$-spaces \cite{Vogt}  and let $\cU$
denote the full subcategory of compactly generated spaces (i.e.\ weakly
Hausdorff $k$-spaces).   Let $B$ be an object of $\cU$ which will remain fixed 
throughout the paper.  We will work 
in the category $\cKB$ of spaces over $B$; an object of $\cKB$ is thus a space 
$X$ together with a map $X\to B$ (the {\em structure map}), while a morphism 
is a map of the underlying spaces which is compatible with the structure maps.  There is a 
natural homotopy theory associated to the category $\cKB$, this is described 
by the $f$-model structure of May and Sigurdsson, recalled in Theorem~\ref{may sig model thm} 
below.    

We will be interested in groups in the category $\cKB$.  In fact there are four notions of 
(internal) group that will play a role in this paper: groups in the category $\cK$ (which we will refer to 
as {\em groups}), groups in the category $s\cK$ of simplicial objects in $\cK$ (which we 
will refer to as {\em simplicial groups}), groups in the category $\cKB$ (which we will refer to 
as {\em parametrized groups}) and groups in the category $s\cKB$ (which we will refer to 
as {\em simplicial parametrized groups}).  In each case, it should be clear from the ambient 
category with respect to which we are working internal to, which of these labels for group objects applies.  
For each of the four notions of group, we have a corresponding notion of principal 
bundle.  For example, we have parametrized 
principal bundles (Definition~\ref{def:fiberwise bundle}) 
and simplicial parametrized principal bundles 
(Definition~\ref{def:simp param principal bundle}).  

The main result of this paper is the construction of a product-preserving classifying 
space functor for parametrized groups, together 
with a corresponding classification theorem for parametrized principal 
bundles.  If $G$ is a parametrized group, we will denote by $BG$ the fiberwise geometric 
realization of the standard simplicial model (see \cite{May-Classifying, Segal-IHES}; 
see also the description following Definition~\ref{def:classifying complex}) 
for the classifying space of $G$.  Our main result states 
that $BG$ is a classifying space for parametrized principal $G$-bundles.  One advantage of this 
result over previous constructions of classifying spaces in the parametrized setting (see for instance 
\cite{Crabb-James}) is that the classifying space functor $B(-)$ so defined is product-preserving.

\begin{restatable}{theorem}{classifyingthm}
\label{classifying thy}

	Let $M$ be a paracompact space over $B$ and let $G$ be a well-sectioned
	fibrant parametrized group.  Then there is a bijection
	\[ 
		H^1(M,G)_{\cKB}\simeq [M,BG]_{\cKB}.  
	\]

\end{restatable} 

Here if $X$ and $Y$ are spaces over $B$, we 
denote by $[X,Y]_{\cKB}$ the set of fiberwise homotopy 
classes of maps from $X$ to $Y$ 
(see Section~\ref{par model str}), 
and $H^1(M,G)_{\cKB}$ denotes the set of 
isomorphism classes of parametrized principal 
$G$-bundles on $M$ (see Section~\ref{sec: princ bundles}).    
We now explain the hypotheses in the theorem above.  
The assumption that $G$ is well-sectioned 
(Definition~\ref{def:well-sectioned group}) 
is the parametrized analog of the 
notion of well-pointed group, which is a standard 
hypothesis to impose in the analogous construction 
of a classifying space for a group in the topological setting.  
One new feature here is that we must also 
impose a fibrancy condition on our parametrized 
groups; namely we must require that they are fibrant 
objects in the model structure of 
Theorem~\ref{may sig model thm}, and we then refer to fibrant 
parametrized groups.  
This is necessary so that, among other things, the projection 
maps of principal bundles are fibrations.


There is a fiberwise geometric realization functor $|-|\colon s\cKB\to \cKB$ sending simplicial parametrized spaces to 
parametrized spaces.  We shall see, in Lemma~\ref{lem:geom realizn of gp object}, that in analogy 
with the corresponding results for ordinary geometric realization, the fiberwise geometric 
realization of a simplicial parametrized group $G$ is a parametrized group $|G|$.  
More generally, we shall prove the following technical theorem    
which asserts that fiberwise geometric realization sends a 
large class of simplicial parametrized principal bundles to ordinary parametrized principal bundles; this 
theorem is a key ingredient in the proof of Theorem~\ref{classifying thy}.  

\begin{restatable}{theorem}{mainresult}
\label{main result}

	Let $G$ be a fibrant simplicial parametrized group and let $M$ be a proper
	simplicial object in $\cKB$.   If $P$ is a simplicial principal bundle
	over $M$  with structure group $G$  such that $P_n\to M_n$ is a
	numerable,  parametrized principal $G_n$-bundle in $\cKB$ for all $n\geq 0$, then the induced
	map
	\[ 
		|P|\to |M| 
	\] 
	on  fiberwise geometric realizations is the projection map for a locally trivial parametrized
	principal $|G|$-bundle $|P|(|M|,|G|)$ in $\cKB$.  Moreover, if the bundle $P_n\to M_n$ is trivial 
	for all $n\geq 0$, then $|P|\to |M|$ is numerable. 

\end{restatable}

Here by a fibrant \emph{simplicial} parametrized group, we mean one for which the parametrized 
groups of $n$-simplices are fibrant for all $n\geq 0$. 
By a {\em proper} simplicial object in $\cKB$ we just mean 
the obvious generalization of the classical notion of 
proper simplicial space \cite{May-GILS} to the parametrized 
setting (see Definition~\ref{def:proper simplicial obj}).  
A standard argument (see Appendix~\ref{app:good implies proper}) 
shows that every {\em good} simplicial object in 
$\cKB$, i.e.\ one whose degeneracy maps are cofibrations in the 
$f$-model structure, is 
automatically proper.  

A prime example of a simplicial principal bundle is the classical notion of 
\emph{principal twisted cartesian product} internal to the category $\cKB$ of parametrized spaces 
(see Section~\ref{sec: simplicial bundles}).  In particular, if $G$ is a simplicial parametrized group then 
we may consider the universal principal twisted cartesian product $WG\to \overline{W}G$.  
The following proposition gives a criterion on $G$ which ensures that $\overline{W}G$ 
is proper and hence satisfies the hypotheses of Theorem~\ref{main result}.   
 
 \begin{restatable}{proposition}{wellpointedprop}
\label{well pointed simp grp} 

	Let $G$ be a well-sectioned simplicial parametrized group.  Then the following statements are true: 
	\begin{enumerate} 
	\item $G$ is a good simplicial group in $\cKB$.
	\item $\overline{W}G$ is proper in $s\cKB$.
	\item $|G|$ is a well-sectioned group in $\cKB$.
	\end{enumerate} 

\end{restatable}

Using Proposition~\ref{well pointed simp grp} we can easily establish that the hypotheses of 
Theorem~\ref{main result} are met for the universal principal twisted cartesian product 
$WG\to \overline{W}G$.  Hence we obtain the following result.   

\begin{restatable}{proposition}{WGcorr} 
\label{corr}

	Let $G$ be a well-sectioned fibrant simplicial parametrized group.  Then the 
	fiberwise geometric realization $|WG|\to |\overline{W}G|$ of the 
	universal $G$-bundle $WG\to \overline{W}G$ is a numerable parametrized principal $|G|$ bundle.  
	Moreover $|WG|$ is a fiberwise contractible group in $\cKB$ containing $|G|$ as a closed subgroup.  

\end{restatable}

Our motivation comes from recent interest in higher principal  bundles or
gerbes \cite{Bartels, JL, Murray,Roberts_PhD, Schreiber_Habilitation, Ste,
Wockel}.   
Recall that for a paracompact space $M$, there is a bijection between
$H^3(M,\ZZ)$ and the set of equivalence classes of $S^1$-\emph{bundle gerbes}
on $M$. An $S^1$-bundle gerbe on $M$ is, roughly speaking, a principal bundle
on $M$ where the structure group $S^1$ is replaced by the 
simplicial topological group $\overline{W}S^1$. From another point of view,
$H^3(M,\ZZ)$ parametrizes the set of isomorphism classes of principal $K(\ZZ,2)$
bundles on $M$. The process of passing from a simplicial principal bundle
for $\overline{W}S^1$ to a principal $K(\ZZ,2)$ bundle can be viewed as an instance of 
our Theorem~\ref{main result} (recall the geometric realisation of $\overline{W}S^1$ is a $K(\ZZ,2)$).
  
Our interest lies in a generalization of this, namely when  the simplicial
group $\overline{W}S^1$ is replaced by an arbitrary simplicial parametrized topological group $G$ 
(subject to quite minor topological conditions) and
we consider simplicial principal bundles with structure group $G$ on $M$.
The resulting set of equivalence classes is isomorphic to the {\em non-abelian
cohomology set} $H^1(M,G)$.  In this case the process of geometric realization
produces  an ordinary principal $|G|$ bundle from a simplicial principal $G$
bundle and therefore  gives rise to a map $H^1(M,G)\to H^1(M,|G|)$.  In
\cite{Ste2}, based on the results of this  paper, the second author proves that this map is an
isomorphism provided that $M$ is paracompact and  $G$ satisfies some mild
topological conditions.

In outline then this paper is as follows.  In Section~\ref{par model str} we review the 
homotopy theory of parametrized spaces from \cite{MaySig}.  In Section~\ref{sec:param_gp} 
we specialize our discussion to parametrized groups and we follow this in 
Section~\ref{sec: princ bundles} with a discussion of principal bundles in the parametrized setting.  
In Section~\ref{sec: simplicial bundles} we consider simplicial parametrized bundles 
and in particular the classical notion of principal twisted cartesian product.  Section~\ref{sec:geom_realzn} 
contains the detailed statements of our main results and the proof of Theorem~\ref{classifying thy}.  
Sections~\ref{app: proof of main result} and~\ref{app:proof of propn 3} contain the 
proofs of Theorems~\ref{main result} and~\ref{well pointed simp grp} respectively, while 
Appendix~\ref{app:good implies proper} is devoted to 
a discussion of the relation between good and proper simplicial objects.

\subsubsection*{Acknowledgements} 

We thank the referee for their very helpful comments which have greatly improved 
the structure and readability of the paper.  
We would also like to thank another anonymous referee for some very useful comments on an 
earlier version of this paper.  
DS would like to thank Tom Leinster for some useful conversations, and 
Urs Schreiber for many email discussions and encouragement.

\section{Parametrized spaces} 
\label{par model str}

In this section we recall some of the basic notions of parametrized homotopy 
theory from \cite{MaySig}; in particular we recount some of the details of the 
$f$-model structure on the category $\cKB$ of spaces over $B$.  

Recall from \cite{MaySig} that
$\cKB$ is a {\em topological bicomplete} category,  in the sense that $\cKB$
is enriched over $\cK$, the underlying category is complete and  cocomplete,
and that it is tensored and cotensored over  $\cK$. 
For any space $K$ and
space $X$ over $B$ the tensor $X\otimes K$ is  defined to be the space
$X\times K$ in $\cK$, considered as a space over $B$ via the obvious map
$X\times K\to B$.  In the sequel, we will often denote the 
tensor $X\otimes K$ simply as $X\times K$.  Similarly, the cotensor $X^K$ is defined to be the space
$\Map_B(K,X)$ given by the pullback square
\[ 
	\xymatrix{ 
	\Map_B(K,X) \ar[r] \ar[d] & \Map(K,X) \ar[d] \\ 
	B \ar[r] & \Map(K,B) } 
\] 
in $\cK$, where the map $B\to \Map(K,B)$ is the conjugate of the projection $B\times K\to B$.
Recall also (see \cite{MaySig}) that $\cKB$ is cartesian closed under the
fiberwise cartesian product $X\times_B Y$ and the fiberwise mapping space
$\Map_B(X,Y)$ over $B$.   The definition of the fiberwise mapping space
$\Map_B(X,Y)$  is rather subtle and we will not give it here, we instead refer
the reader to Definition~1.3.7 of \cite{MaySig}.  

Since $\cKB$ is a topological bicomplete category 
there is a natural notion of geometric realization for simplicial objects 
in $\cKB$ -- the notion of {\em fiberwise geometric realization}.  
If $X$ is a simplicial object in $\cKB$, i.e.\ a parametrized simplicial space, then the fiberwise
geometric realization $|X|$ of  $X$ is defined by the usual coend formula:
\[ 
	|X| = \int^{[n]\in \Delta} X_n\times \Delta^n.
\]
In other words, one regards $X$ as a simplicial object in $\cK$ and  computes
the ordinary geometric realization, and then  one equips this with the induced
map to $B$.   In particular, $|X|$ is obtained as a quotient from the coproduct 
$\sqcup_{n\geq 0} X_n\times \Delta^n$.  

It follows from the non-parametrized case that fiberwise geometric realization gives rise to a co-continuous functor 
$|\cdot|\colon s\cKB\to \cKB$.  Since ordinary geometric  realization commutes with
finite limits, fiberwise  geometric realization also commutes
with finite limits in $\cKB$, and moreover is compatible with
the topological structures on $s\cKB$ and $\cKB$ in the sense that  $|X\times K| =
|X|\times K$ for any space $K$ in $\cK$.  Note also that the 
fiberwise geometric realization of a level-wise closed inclusion is a 
closed inclusion.

In \cite{MaySig} several model structures on  $\cKB$ are introduced.  The
model structure on $\cKB$ that we will be interested in  has its origins in
the work \cite{SV} of Schw\"{a}nzl and Vogt.  In \cite{SV}   (see also
\cite{Cole} and \cite{MaySig})  the authors consider a topological bicomplete
category $\cC$ and  define three classes of morphisms: $h$-equivalences,
$h$-fibrations and  $\bar{h}$-cofibrations.  A morphism $f\colon X\to Y$ in
$\cC$ is an  $h$-equivalence if and only if it is a homotopy equivalence,
defined in  terms of the cylinder object $X\times I$ where
$I$ denotes the unit interval.  A morphism $f\colon X\to Y$ is called an
$h$-fibration if and only if it has the RLP (right lifting property) with
respect to all  morphisms of the form $Z\times \{0\}\to Z\times I$, while
$f$ is  called an $\bar{h}$-cofibration if and only if  $X\times
I\cup_{X\times \{0\}} Y\times \{0\} \to Y\times I$ has the  LLP (left lifting property) with respect to all
$h$-fibrations in $\cC$.

In \cite{SV} the $\bar{h}$-cofibrations are called {\em strong cofibrations}
and the following alternative characterization of them is given: a morphism
$f\colon X\to Y$ is an  
$\bar{h}$-cofibration if and only if it has the LLP  
with respect to all $h$-fibrations which are also $h$-equivalences -- i.e.\
the $h$-acyclic $h$-fibrations. When $\cC = \cK$, the class of strong
cofibrations equals the class of closed cofibrations. Under suitable
hypotheses on $\cC$ (see Theorem 4.2 of \cite{Cole} and Theorem 4.2.12 of
\cite{MaySig}; see also \cite{Barthel-Riehl}) these three classes of morphisms equip $\cC$ with the structure
of a proper, topological model category. This model structure is sometimes
called the {\em $h$-model structure}.

If we specialize to the case when $\cC = \cKB$, it 
turns out (see \cite{Barthel-Riehl, Cole, MaySig}) that the
required hypotheses are satisfied and the above
notions of $h$-equivalence, $h$-fibration and $\bar{h}$-cofibration  equip
$\cKB$ with the structure of a model category.   This model structure is
called the $f$-model structure (for fiberwise) and the weak equivalences,
fibrations and cofibrations are labelled accordingly.   A precise statement is
the following.

\begin{theorem}[May-Sigurdsson \cite{MaySig}, Theorem 5.2.8] 
\label{may sig model thm}

$\cKB$ has the structure of a proper, topological model category for which 

	\begin{itemize} 

		\item the weak equivalences are the $f$-equivalences, 

		\item the fibrations are the $f$-fibrations, 

		\item the cofibrations are the $\bar{f}$-cofibrations.  

	\end{itemize} 

\end{theorem} 

Recall that a model category $\cC$ is said to be {\em topological} if it  is a
$\cK$-model category in the sense of Definition 4.2.18 of \cite{Hovey},  for
the monoidal model structure on $\cK$ given by the above  $h$-model structure
(observe that this coincides with the classical  Str\o m model structure
\cite{Cole,MaySig,Strom3} on $\cK$).

To be completely explicit, we explain the labels on the  three classes of maps
in the above theorem. A map $g\colon X\to Y$ in $\cKB$ is called an $f$-{\em
equivalence} if it is a fiberwise homotopy equivalence.  This needs the notion
of homotopy over $B$, which is formulated in terms of  $X\times I$. A map
$g\colon X\to Y$ in $\cKB$ is called an $f$-{\em fibration}  if it has the
fiberwise covering homotopy property, i.e.\ if it has the  RLP property with
respect to all maps of the form  $i_0\colon Z\to Z\times I$ for all $Z\in
\cKB$. A map $g\colon X\to Y$ in $\cKB$ is called an  $\bar{f}$-{\em
cofibration}, or a {\em strong cofibration}  if it has the LLP property with
respect to all $f$-acyclic $f$-fibrations. There is also the notion of an
$f$-{\em cofibration}: this is a map $g\colon X\to Y$ which satisfies the LLP
property with respect to all maps of the form  $p_0\colon \Map_B(I,Z)\to Z$
for some $Z\in \cKB$. Every $\bar{f}$-cofibration $g\colon X\to Y$ in $\cKB$
is an $f$-cofibration. The converse is not true in general.  However May  and
Sigurdsson prove (see Theorems 4.4.4 and 5.2.8 of \cite{MaySig}) that if
$g\colon X\to Y$ is a \emph{closed} $f$-cofibration then $g$ is an
$\bar{f}$-cofibration.

Moreover, in analogy with the standard characterization of  closed Hurewicz
cofibrations in terms of NDR pairs, May and  Sigurdsson give a criterion (see
Lemma 5.2.4 of \cite{MaySig})  which detects when a closed inclusion $i\colon
A\to X$ in  $\cKB$ is an $\bar{f}$-cofibration.   Such an inclusion $i\colon
A\to X$ in $\cKB$ is an  $\bar{f}$-cofibration if and only if $(X,A)$ is a
{\em fiberwise NDR pair} in the sense that there is a map  $u\colon X\to I$
for which $A = u^{-1}(0)$ and a homotopy $h\colon X\times I\to X$  over $B$
such that $h_0 = \mathit{id}$, $h_t|_A = \mathit{id}_A$  for all $0\leq t\leq
1$ and $h_1(x)\in A$ whenever $u(x) < 1$

\section{Parametrized groups} 
\label{sec:param_gp}

In this section we study the four classes of groups described in the introduction: groups, parametrized groups, simplicial 
groups and simplicial parametrized groups, corresponding to group 
objects in $\cK$, $\cKB$, $s\cK$ and $s\cKB$ respectively.  

As a group object in $\cKB$, a parametrized group has a natural structure as an {\em ex-space}, i.e.\ a space 
$X$ over $B$ equipped with a section of the structure map $X\to B$ (see Section 1.3 of 
\cite{MaySig} for more details).  
For such a parametrized group $G$, the structure as an ex-space arises from the 
canonical section of the structure map of $G$ given by the identity section.  In the context of 
parametrized spaces, ex-spaces are the analog of pointed spaces in the non-parametrized 
setting.  The analog of a well-pointed, or non-degenerately based space, is the notion of a 
well-sectioned ex-space, i.e.\ one for which the distinguished section of the structure map 
is an $\bar{f}$-cofibration.  In particular 
the ex-space analog of a well-pointed group is the notion of a well-sectioned parametrized group 
in the sense of the following definition.  

\begin{definition} 
\label{def:well-sectioned group} 
Let $G$ be a parametrized group.  We say that $G$ is {\em well-sectioned} 
if the identity section $1_G\colon B\to G$ is an $\bar{f}$-cofibration.  We say that 
a simplicial parametrized group is {\em well-sectioned} if the parametrized group of 
$n$-simplices is well-sectioned for every $n\geq 0$.  
\end{definition}

We shall also need to impose a fibrancy condition on parametrized groups.  
Accordingly, we make the following definition.  

\begin{definition} 
\label{def:fibrant param gp} 
Let $G$ be a parametrized group.  We say that $G$ is {\em fibrant}, if $G$ is $f$-fibrant 
considered as an object of $\cKB$.   We say that a simplicial parametrized group is 
{\em fibrant} if it is level-wise fibrant in the sense that $G_n$ is fibrant for all 
$n\geq 0$.  
\end{definition} 

We shall see that in order to obtain a notion of parametrized principal $G$-bundle 
(Definition~\ref{def:fiberwise bundle} below) that is well-behaved homotopically 
in the sense that is a $f$-fibration, then we need to impose the condition that 
$G$ is fibrant (see Theorem~\ref{thm:CJ locally trivial} below).   

Recall from Section~\ref{par model str} above, that the fiberwise geometric realization functor $|\cdot |\colon 
s\cKB\to \cKB$ preserves products.   Hence we have the following 
obvious Lemma.  

\begin{lemma}
\label{lem:geom realizn of gp object} 
The fiberwise geometric realization functor $|-|\colon s\cKB\to \cKB$ sends group 
objects in $s\cKB$ to group objects in $\cKB$, in other words, if $G$ is a simplicial 
parametrized group then $|G|$ is a parametrized group.  
\end{lemma}

If $G$ is a parametrized group then there is a natural notion of a
$G$-space over $B$ and a $G$-map between $G$-spaces over $B$.  A $G$-space
over $B$ is a  space $X$ over $B$ equipped with an action of $G$, i.e.\ a map
$X\times_B G\to X$ of spaces over $B$ making the usual diagrams commute, and a $G$-map 
from $X$ to $Y$ is a map $X\to Y$ in $\cKB$ compatible with the respective 
$G$-actions.  
We write $G\cKB$ for the category consisting of
$G$-equivariant objects and $G$-maps between them.  We have the following
lemma.

\begin{lemma} 
\label{GK/B top bicomplete}

	The category $G\cKB$ is a topological bicomplete category.  

\end{lemma} 

\begin{proof}

To construct limits in $G\cKB$ one first constructs the corresponding limit
in $\cKB$  and then equips it with the induced $G$-action.  To construct
colimits in $G\cKB$ one first constructs the colimit in $\cKB$ and then one
observes that, since $G\times_B (-)$ is a left adjoint and therefore preserves
colimits,  the colimit in $\cKB$ comes equipped with a natural $G$-action.
The category $G\cKB$ is naturally enriched over $\cK$; if $X$ and $Y$ are objects of
$G\cKB$ then the space of morphisms $G\cKB(X,Y)$ 
is given by the equalizer diagram
\[ 
	G\cKB(X,Y)\to \cKB(X,Y)\rightrightarrows \cKB(X\times_B G,Y)  
\]
in $\cK$, where the last two maps are induced by the actions of $G$ on 
$X$ and $Y$ respectively, i.e.\ the maps which send 
a map $f\colon X\to Y$ in $\cKB$ to the compositions $X\times_B G\xrightarrow{f\times_B 1_G} Y\times_B G 
\to Y$ and $X\times_B G\to X\xrightarrow{f} Y$. 
If $X\in G\cKB$ and $K\in \cK$ then the tensor
$X\otimes K$ is the usual one in $\cKB$ equipped with the $G$-action where
$G$ acts trivially on the $K$ factor. The cotensor in $G\cKB$ is the usual
cotensor in $\cKB$ equipped with an action of $G$ described as follows. The
commutative diagram
\[ 
	\xymatrix{ 
		G \ar[d] \ar[r] & \Map(K,G) \ar[d] \\ 
		B \ar[r] & \Map(K,B), 
	} 
\] 
where the top horizontal map is the adjoint of the projection $G\times K\to G$ in $\cK$,
shows that there is a natural morphism $G\to \Map_B(K,G)$ in $\cKB$.  The
action of $G$ on $\Map_B(K,X)$ is given by the following composite:
\[ 
	\Map_B(K,X)\times_B G \to \Map_B(K,X)\times_B \Map_B(K,G) 
	\to \Map_B(K,X),  
\] 
where the second map is induced by the action of $G$ on $X$ via the identification 
\[
	\Map_B(K,X)\times_B \Map_B(K,G) \cong \Map_B(K,X\times_B G).
\]
One can check that this gives a $G$-action as claimed.  To check that we have
required adjunction homeomorphisms, observe that we have the following
isomorphisms  of diagrams in $\cK$:
\[ 
	\xy 
	(-41,0)*+{\cKB(X\times K,Y)}; 
	(-19,0)*+{\cong}; 
	(0,0)*+{\cKB(X,Y^K)};
	(18.5,0)*+{\cong};
	(42,0)*+{\cK(X,\cKB(X,Y))};
	(-41,-15)*+{\cKB((X\times_B G)\times K,Y)}; 
	(-19,-15)*+{\cong}; 
	(0,-15)*+{\cKB(X\times_B G,Y^K)};
	(18.5,-15)*+{\cong};
	(42,-15)*+{\cK(K,\cKB(X\times_B G,Y))};
	{\ar (-42,-3)*{};(-42,-12)*{}};
	{\ar (-40,-3)*{};(-40,-12)*{}};
	{\ar (-1,-3)*{};(-1,-12)*{}};
	{\ar (1,-3)*{};(1,-12)*{}};
	{\ar (41,-3)*{};(41,-12)*{}};
	{\ar (43,-3)*{};(43,-12)*{}};
	\endxy
\]
where we have used the fact that we have an isomorphism $(X\times K)\times_B
G\cong (X\times_B G)\times K$.  Therefore, on forming equalizers we get the
required natural isomorphisms
\[
	G\cKB(X\times K,Y)\cong G\cKB(X,Y^K)\cong \cK(K,G\cKB(X,Y)),  
\]
using the fact that $\cK(K,-)$ preserves equalizers.
\end{proof} 

Let $G$ continue to denote a group object in $\cKB$.   In $G\cKB$  there are natural notions of
$f$-equivalence, $f$-fibration, $f$-cofibration and $\bar{f}$-cofibration.
Thus a map $g\colon X\to Y$ in $G\cKB$ is an $f$-cofibration if it has the
LLP in $G\cKB$ with respect to  $G$-maps of the form $p_0\colon
\Map_B(I,Z)\to Z$ for all $Z$ in  $G\cKB$. Similarly, we say that a map
$g\colon X\to Y$ in $G\cKB$ is an $f$-equivalence if it is a fiberwise
$G$-homotopy equivalence.  A map $g\colon X\to Y$ in $\cKB$ is  an
$f$-fibration if it has the RLP in $G\cKB$ with respect to $G$-maps of the
form $i_0\colon Z\to Z\times I$ for all $Z$ in $G\cKB$.   A map $g\colon
X\to Y$ in $G\cKB$ is an $\bar{f}$-cofibration if it has the LLP in $G\cKB$
with respect to all $f$-acyclic $f$-fibrations in $G\cKB$.

Just as above, there is a criterion to detect when an inclusion $i\colon A\to
X$ in $G\cKB$ is an $\bar{f}$-cofibration in $G\cKB$.  We have the  following
result which will play a key role in the proof of Theorem~\ref{main result}.

\begin{lemma} 
\label{ndr}

	An inclusion $i\colon A\to X$ in $G\cKB$ is an $\bar{f}$-cofibration if
	and only if $i(A)$ is closed in $X$ and there is a representation of
	$(X,A)$ as a $G$-fiberwise NDR pair.

\end{lemma} 

Here by a representation of $(X,A)$ as a $G$-fiberwise NDR pair we understand,
in analogy with \cite{Steenrod}, that there is a pair  $(u,h)$ of maps with
$u\colon X\to I$ and $h\colon X\times I\to X$ which represent $(X,A)$ as a
fiberwise  NDR pair and which satisfy $u(xg) = u(x)$ for all $x\in X$ and
$g\in G$, and $h(xg,t) = h(x,t)g$ for all  $(x,t)\in X\times I$ and $g\in
G$.

\begin{proof} 

We will explain how to adapt Steps 1--3 in the proof of Theorem 4.4.4  of
\cite{MaySig} to our setting.  Step 3 adapts in a straightforward  way to show
that $i(A)$ is closed in $X$: factor the inclusion $i\colon A\to X$ as $A\to
E\to X$ where  $E = A\times I \cup X\times (0,1]$ and where  $i_0\colon
A\to E$ is given by $i_0(a) = (a,0)$.  Analogous to  the corresponding
statement in  \cite{MaySig}, the projection $\pi\colon E\to X$ is an
$f$-acyclic  $f$-fibration in $G\cKB$. Therefore there exists a  map
$\lambda\colon X\to E$ extending $i$, i.e.\ $\lambda\circ i = i_0$.   Let
$\psi\colon E\to I$ denote the projection onto the second factor and  note
that $\psi^{-1}(0) = i_0(A)$, so that $i_0(A)$ is closed in $E$.   Therefore
$\lambda^{-1}\left(i_0(A)\right) = i(A)$ is closed in $X$ (since $\lambda$ is injective).
Standard  arguments now show that $(X,A)$ has a representation as a
$G$-fiberwise NDR pair.

Next we explain how Steps 1 and 2 can be adapted to show  that if $(X,A)$ has
a representation as a $G$-fiberwise NDR pair  then $i\colon A\to X$ is an
$\bar{f}$-cofibration.  The usual argument shows that  $X\times\{0\}\cup
A\times I$ is a retract of $X\times I$ in  $G\cKB$.  Hence $i\colon A\to
X$ is a closed $f$-cofibration in  $G\cKB$ and so $Mi\to X\times I$ is the
inclusion of a strong  deformation retraction (see Lemma 4.2.5 of
\cite{MaySig}), where $Mi$ is the mapping cylinder of $i$.   The map $u$ in a representation $(u,h)$ of $(X,A)$ as a
$G$-fiberwise  NDR pair can be used to show that there exists a map
$\psi\colon   X\times I\to I$  such that $\psi^{-1}(0)= Mi$.  The analogue
of Theorem 3  of \cite{Strom} for the category $G\cKB$ then shows that $Mi\to
X\times I$  has the LLP with respect to all $f$-fibrations and hence
$i\colon A\to X$ is an $\bar{f}$-cofibration.
\end{proof}

Finally, let us note (\cite{SV} Lemma 2.6) that since $\bar{f}$-cofibrations
in $G\cKB$ are defined by a  left lifting property, the following is true.

\begin{lemma} 
 \label{seq colimit lemma}

	If $X_0\to X_1 \to \cdots \to X_n\to \cdots$ is a sequence of
	$\bar{f}$-cofibrations  in $G\cKB$, then $X_n\to X$ is an
	$\bar{f}$-cofibration in $G\cKB$ for all $n\geq 0$,  where $X = \colim X_n$.

\end{lemma}

\section{Parametrized principal bundles} 
\label{sec: princ bundles}

In this section we study the notion of a principal bundle in $\cKB$ 
for a parametrized group $G$, in other words the notion of a parametrized 
principal bundle.  In particular we study some homotopy-theoretic properties of 
parametrized principal bundles for the homotopy theory of 
Theorem~\ref{may sig model thm}. The notion of parametrized principal 
bundle was introduced in \cite{Crabb-James}, we re-phrase it 
in the following way.

\begin{definition}[\cite{Crabb-James}]  
\label{def:fiberwise bundle}

	Let $G$ be a parametrized group.  A {\em parametrized principal $G$
	bundle} in $\cKB$  consists of a $G$-space $P$ in $\cKB$ together with a map 
	$\pi\colon P\to M$ such that (i) $\pi$ admits local sections (in $\cKB$) and (ii) the square 
	\begin{equation} 
	\label{strongly free} 
		\xy 
		(-9.5,6.25)*+{P\times_B G}="1"; 
		(-9.5,-6.25)*+{P}="2"; 
		(9.5,6.25)*+{P}="3"; 
		(9.5,-6.25)*+{M}="4"; 
		{\ar_-{p_1} "1";"2"}; 
		{\ar "1";"3"};
		{\ar^-{\pi} "3";"4"}; 
		{\ar_-{\pi} "2";"4"}; 
		\endxy 
	\end{equation}
	is a pullback in $\cKB$, where the horizontal map $P\times_B G\to P$ is the action of 
	$G$ on $P$ and the map $p_1$ is projection onto the first factor.  
\end{definition}

 The condition (ii) implies that the action of $G$ on $P$ is principal with 
$M$ as its space of orbits and that the action of $G$ preserves the fibers of $\pi$.  
The condition that $\pi\colon P\to M$ admits local sections 
means that for every point of $m$ of $M$ there is an open 
neighborhood $U_m\subset M$ of $m$ together with a fiberwise map 
$s\colon U_m\to P$ which is a section of $\pi$.

We use the standard terminology: $P$ is the {\em total space}, $M$ is the 
{\em base space} and $G$ is the {\em structure group} of a parametrized principal 
bundle, which we shall sometimes denote 
by $P(M,G)$.  A 
{\em morphism} $P(M,G)\to P'(M',G')$ of parametrized principal bundles consists of a triple of 
maps $f\colon M\to M'$, $\bar{f}\colon P\to P'$ and 
$\alpha\colon G\to G'$ in $\cKB$, where $\alpha$ is a 
homomorphism of parametrized groups and $\bar{f}$ is equivariant for $\alpha$.  
Parametrized principal bundles, together with the morphisms between 
them, form the category of parametrized principal bundles.

We make the following definition.  

\begin{definition}
\label{def:simp param principal bundle}

	A {\em simplicial parametrized principal bundle} is a simplicial object in 
	the category of parametrized principal bundles.  
	
\end{definition} 

Thus if $P(M,G)$ is a simplicial parametrized principal bundle with 
projection map $\pi\colon P\to M$, then each map $\pi_n\colon P_n\to M_n$ is a parametrized principal $G_n$-bundle 
and the face and degeneracy maps are morphisms of parametrized principal bundles.  

It is worth reformulating this definition in a slightly different way.  A simplicial parametrized principal 
bundle consists of a simplicial parametrized group $G$, a  
simplicial parametrized space $P$ equipped with an action 
of $G$ in $s\cKB$ and a map $\pi\colon P\to M$ which satisfies the analogs of (i) and (ii) in 
Definition~\ref{def:fiberwise bundle} above.  Thus the diagram analogous to~\eqref{strongly free} 
is a pullback in $s\cKB$ and the map $\pi$ admits local sections level-wise 
in the sense that $\pi_n\colon P_n\to M_n$ admits local sections for all $n\geq 0$. Note that we do not require any compatibility between these local sections and the face and degeneracy maps for the simplicial spaces.

Recall from Section~\ref{par model str} that the fiberwise geometric realization functor 
$|\cdot |\colon s\cKB\to \cKB$ preserves finite limits.  It follows that if $P(M,G)$ 
is a simplicial parametrized principal bundle then there is an induced action of $|G|$ (see 
Lemma~\ref{lem:geom realizn of gp object}) on $|P|$ 
in $\cKB$ such that the diagram 
\[
	\xymatrix{ 
		|P|\times_B |G| \ar[d]_-{p_1} \ar[r] & |P| \ar[d]^-{|\pi|} \\ 
		|P| \ar[r]_-{|\pi|} & |M| 
	} 
\]
is a pullback in $\cKB$.  In Theorem~\ref{main result} we will give conditions on 
$M$ and $G$ which ensure that the map $|\pi|\colon |P|\to |M|$ has local sections 
and hence that $|P|(|M|,|G|)$ is a parametrized principal bundle.  

If we consider morphisms of parametrized principal bundles with fixed structure 
group $G$ and fixed base space $M$ (both parametrized, of course), then,  
just as for ordinary principal bundles, every such morphism is an
isomorphism.  We denote the set of isomorphism classes of  parametrized principal
$G$-bundles on $M$ by $H^1(M,G)_{\cKB}$.

Every parametrized principal $G$-bundle $\pi\colon P\to M$ is a parametrized fiber
bundle in the sense that each point of $M$ has an open neighborhood $U$ such
that  the restriction of $P$ to $U$ is isomorphic to the trivial parametrized $G$-bundle $U\times_B G$. 
If $G$ is fibrant, such a trivial parametrized fiber bundle is an
$f$-fibration in the sense of Theorem~\ref{may sig model thm}.  When $B$ is a
point it is a well known theorem  that every numerable fiber bundle $E\to M$
is a Hurewicz fibration. There is an obvious extension of this notion to the
notion of a numerable parametrized fiber bundle: a parametrized fiber bundle is
{\em numerable} if it is fiberwise locally trivial relative to a numerable open
cover of the base space. We have the following theorem from 
\cite{Crabb-James}.

\begin{theorem}[\cite{Crabb-James}] 
\label{thm:CJ locally trivial}
	Let $p\colon E\to M$ be a map in $\cKB$. Suppose that $p^{-1}V_i\to V_i$
	is an $f$-fibration for each open set $V_i$ in a numerable covering
	$(V_i)_{i\in I}$ of $M$.   Then $p$ is an $f$-fibration.  In particular, if $G$ is a 
	fibrant parametrized group, then 
	any parametrized principal $G$-bundle $\pi\colon P\to M$ in $\cKB$ over a
	paracompact base space $M$, or more generally  any numerable parametrized 
	principal $G$ bundle in $\cKB$, is an $f$-fibration.

\end{theorem} 

This theorem has the following important corollary.   In the parametrized
context, principal $G$-bundles $P_0$ and $P_1$ on $M$ are said to be {\em
fiberwise concordant}  if there exists a parametrized principal $G$-bundle $P$ on $M\times
I$ together with fiberwise isomorphisms $P_0\cong P|_{M\times \{0\}}$ and
$P_1\cong P|_{M\times \{1\}}$.  The fiberwise concordance  relation is clearly
an equivalence relation.  When $B$ is a point it is well known that there is a
bijection between the set of isomorphism classes of numerable principal $G$
bundles on $M$ and concordance  classes of principal $G$ bundles on $M$.  
From Theorem~\ref{thm:CJ locally trivial}, we see that in
the parametrized setting there is an analogous bijection.

\begin{corollary}  
\label{concordance} 

	Let $M$ be a paracompact space in $\cKB$ and let $G$ be a fibrant parametrized group.  
	Then there is a bijection between $H^1(M,G)_{\cKB}$ and the set of
	fiberwise concordance classes of parametrized principal $G$-bundles on $M$.

\end{corollary} 

\begin{proof}

To prove that there is such a bijection one needs to know that fiberwise
concordant bundles are isomorphic.  
For this, it is enough to prove that there is an isomorphism $P\cong P_0\times I$, when $P$ is a 
parametrized principal $G$-bundle on $M\times I$, and $P_0$ denotes the restriction to $M\times \{0\}$.
Consider the bundle $P\times_G
(P_0\times I)$ on  $M\times I$.  There is a section of this bundle over the
closed subspace  $M\times \{0\}$ of $M\times I$.  We want to know that
this section extends to a section  defined over $M\times I$.  Since
$P\times_G (P_0\times I)$ is a fiberwise locally  trivial bundle on $M\times
I$, it is an $f$-fibration.  Therefore the required  extension of the section
exists, since the inclusion $M\times \{0\}\subset  M\times I$ is an
$f$-acyclic $\bar{f}$-cofibration.  It follows that the set  of fiberwise
concordance classes of $G$-bundles on $M$ is isomorphic to $H^1(M,G)_{\cKB}$.
\end{proof}

We shall also need the following result, related to 
Theorem 12 of \cite{Strom2}.  
\begin{proposition}
\label{Stroms theorem}

	Let $\pi\colon P\to M$ be a numerable parametrized
	principal $G$ bundle for a fibrant parametrized group $G$ 
	and suppose that $A\subset M$ is a closed  inclusion which is an
	$\bar{f}$-cofibration in $\cKB$.  Then the closed inclusion $P|_A\subset
	P$ is an $\bar{f}$-cofibration in  $G\cKB$.

\end{proposition} 

\begin{proof} 

The proof of the analogous result in \cite{Strom2} can be adapted to this
setting as follows.  Choose a representation $(u,h)$ of  $(M,A)$ as a
fiberwise NDR pair in $\cKB$.  Next observe that in the diagram
\[
	\xymatrix{ 
		P \ar[d]_-{i_0} \ar[r] & P \ar[d]^-{\pi} \\ 
		P\times I \ar[r]_-{h(\pi\times 1)} \ar@{.>}[ur]^-{\bar{h}} & M 
	} 
\]
the indicated lifting $\bar{h}$ can be found, and moreover can be chosen to be
$G$-equivariant, in  light of the proof of Corollary~\ref{concordance} above.
To finish the proof, we need to show  that we can choose $\bar{h}$ so that
$\bar{h}(x,t) = x$ for any $x\in P|_A$.  Consider the  associated bundle
$\Aut_0(P\times I) = (P\times I)\times_G G$ on $M\times I$, where the 
action of $G$ on itself  is conjugation.  Note that sections of
$\Aut_0(P\times I)$ are bundle automorphisms of $P\times I$  covering the
identity on $M\times I$.  Since $\pi \bar{h} = h(\pi\times 1)$ and $\bar{h}$
is equivariant, it follows that  $\bar{h}$ restricts to a section of
$\Aut_0(P\times I)$ over $A\times I\subset M\times I$.  Similarly  the
restriction of $\bar{h}$ to $P\times \{0\}$ defines a section of $\Aut(P\times
I)$ over $M\times \{0\}$.   Since $\Aut_0(P\times I)\to M\times I$ is a
locally trivial, numerable, parametrized bundle, and $(A\times I)\cup (M\times
\{0\})  \subset M\times I$ is a closed $\bar{f}$-cofibration, it follows that
we can  find the indicated lifting in the diagram
\[
	\xymatrix{ 
	(A\times I)\cup (M\times \{0\}) \ar[d] \ar[r] & \Aut_0(P\times I) \ar[d] \\ 
	M\times I \ar[r]_-1 \ar@{.>}[ur]^-{\bar{k}} & M\times I. } 
\]
Now define $\tilde{h} = \bar{h}\bar{k}^{-1}$.  Then $\tilde{h}\colon P\times
I\to P$ is  $G$-equivariant and satisfies $\pi \tilde{h} = h(\pi\times 1)$.
If we set $\tilde{u} = u\pi$  then it is easily checked that
$(\tilde{u},\tilde{h})$ is a representation of  $(P,P|_A)$ as a
$G$-equivariant NDR pair.
\end{proof}

\section{Simplicial principal bundles and twisted cartesian products} 
\label{sec: simplicial bundles}

In this section we recall the notion of principal twisted Cartesian product defined 
internally to a category $\cC$ with finite limits, and we recall the definition of 
the universal simplicial $G$-bundle $WG\to \overline{W}G$ associated to a group object $G$ in 
$\cC$.  Recall the following classical definition (see for instance \cite{May-SOAT}).  

\begin{definition}

	Let $G$ be a group in $\sSet$.  A {\em principal twisted cartesian product} with 
	structure group $G$ in 
	$\sSet$ consists of a $G$-simplicial set $P$ and a map $\pi\colon P\to M$ 
	such that $\pi$ has a pseudo-cross section and the diagram analogous to~\eqref{strongly free} 
	above is a pullback.  
	
\end{definition}

By a pseudo-cross section (Definition 18.5 of 
\cite{May-SOAT}) we mean a collection of maps $\sigma_n\colon M_n\to P_n$ for 
all $n\geq 0$ such that $\sigma_i s_i = s_i\sigma_i$ for all $0\leq i\leq n+1$, 
$n\geq 0$, and $\sigma_i d_i = d_i\sigma_i$ for all $0< i\leq n$ and $n\geq 0$.    
We note that a pseudo-cross section can be conveniently reformulated in terms 
of Illusie's d\'{e}calage functor (see \cite{Duskin Mem} and also the discussion below) and that this leads 
to a simple description of the classifying theory of principal twisted cartesian 
products (see \cite{Ste1}).  

It is clear from the preceding discussion that we may replace the category 
$\Set$ of sets with any category $\cC$ with finite limits  and obtain the notion of principal twisted 
 cartesian product {\em internal to the category} $s\cC$ of simplicial objects in $\cC$.  
Of particular interest for us will be the case where $\cC = \cKB$; 
note that principal twisted cartesian products in this case are examples 
of simplicial parametrized principal bundles 
(Definition~\ref{def:simp param principal bundle}).  

The data of a principal twisted cartesian product may be conveniently reformulated 
in terms of {\em twisting functions}, as we now recall.  A family of maps $t_n\colon M_n\to 
G_{n-1}$ defined for $n\geq 1$ is called a {\em twisting function} if the identities ($T$) on 
page 71 of \cite{May-SOAT} are satisfied,  when 
interpreted internally in the obvious fashion.  Every principal twisted 
cartesian product determines a unique twisting function, and conversely a 
twisting function determines a principal twisted cartesian product 
\[
M\times_t G
\] 
in which the object of $n$-simplices is the product $(M\times_t G)_n = M_n\times G_n$, 
and where the face and degeneracy maps are defined as in Definition 18.3 of 
\cite{May-SOAT}.  In particular the description in terms of twisting functions 
explains the origin of the terminology `twisted cartesian product'.    

If $G$ is a simplicial group internal to $\cC$ (for instance a simplicial 
group or a simplicial parametrized group), then the universal $G$ bundle $WG\to \overline{W}G$ 
has a convenient description via twisting functions.  

\begin{definition}
\label{def:classifying complex}
Let $\cC$ be a category with finite limits and let $G$ be a group 
in $s\cC$.  The {\em classifying complex} $\overline{W}G$ is defined to be 
the simplicial object of $\cC$ with $(\overline{W}G)_0 = 1$, the terminal object of $\cC$, and 
\[
(\overline{W}G)_n = G_{n-1}\times \cdots \times  G_0 
\]
for $n\geq 1$, with face and degeneracy maps defined by the following formulae: 
\begin{align*}
& d_i(g_{n-1},\ldots,g_0) = (d_i(g_{n-1}),\ldots, (d_i(g_i))g_{i-1},\ldots,g_{i-2},\ldots,g_0) \\ 
& s_i(g_{n-1},\ldots,g_0) = (s_i(g_{n-1}),\ldots,s_i(g_i),1,g_{i-1},\ldots,g_0), 
\end{align*}
if $(g_{n-1},\ldots,g_0)\in (\overline{W}G)_n$.  
\end{definition}

When $\cC = \Set$ is the category of sets and so $G$ is an ordinary simplicial group, this is the traditional classifying complex 
construction introduced in \cite{Kan}.  In the next section we shall make a more 
careful study of this construction in the case when $\cC = \cKB$.  Note that 
when $G$ is group in $\cC$ and we abusively denote by $G$ the constant 
simplicial group in $\cC$ with all face and degeneracy maps equal to the identity, 
then $\overline{W}G$ reduces to the familiar description in terms of the 
nerve of the one-object groupoid $G$ in $\cC$.  Therefore, in this case 
we have the identification 
\begin{equation}
\label{eq:simp BG}
(\overline{W}G)_n = G\times \cdots \times G\quad (\text{$n$ factors}) 
\end{equation}
with face and degeneracy maps defined by the usual formulae: 
\begin{align} 
& d_i(g_0,\ldots,g_{n-1}) = \begin{cases} 
					(g_1,\ldots,g_{n-1})  \quad \text{if}\ i=0, \\ \label{eq:face maps for BG}
					(g_0,\ldots, g_{i-1}g_i,\ldots, g_{n-1})\quad   \text{if}\ 1\leq i\leq n-1, \\
					(g_0,\ldots,g_{n-2})  \quad \text{if}\ i=n 
					\end{cases} 										\\ 
& s_i(g_0,\ldots,g_{n-1}) = (g_0,\ldots,g_{i-1},1,g_i,\ldots,g_n). \label{eq:degen maps for BG}
\end{align} 
Alternatively, we may think of the one-object groupoid $G$ as the action groupoid 
$1//G$ in $\cC$, associated to the trivial action of $G$ on the terminal object $1$ of $\cC$.  
 Although it will not play an important role in this paper, we mention in passing 
a very useful conceptual approach to the classifying complex construction due to Duskin.  

For every $n\geq 0$, we may form the simplicial object $N(1//G_n)$ which is the nerve of 
the action groupoid $1//G_n$ associated to the group $G_n$; in this way we obtain a 
bisimplicial object $N(1//G)$ in $\cC$.  In the paper \cite{AM}, Artin and Mazur introduced 
the construction of the {\em total simplicial set} $T(X)$ associated to a bisimplicial set $X$.  
This construction makes sense in any category $\cC$ with finite limits and defines a 
functor 
\[
T\colon ss\cC\to s\cC, 
\]
called the total simplicial object functor.  It is not hard to show, using explicit formulas for  
face and degeneracy maps, that there is an isomorphism 
\[
\overline{W}G = T(N(1//G)) 
\]
of simplicial objects in $\cC$.  Besides the conceptual understanding that this 
observation brings to the classifying complex construction, it also gives a useful perspective 
on the construction of the universal principal twisted cartesian product over 
$\overline{W}G$.  

The right action of $G$ on itself defines an action groupoid $G//G$ in $\cC$; there 
is a natural functor $G//G\to 1//G$ and hence a simplicial map 
\begin{equation}
\label{eq:T(N) bundle} 
T(N(G//G))\to T(N(1//G)) 
\end{equation} 
on taking nerves and applying the total simplicial object functor.  It is straightforward 
to see that there is a canonical action of the simplicial group $G$ on 
$T(N(G//G))$ such that the diagram analogous to~\eqref{strongly free} above 
is a pullback.  With a little more work, exploiting the close relationship between 
the functor $T$ and Illusie's d\'{e}calage functor, one may show that the map~\eqref{eq:T(N) bundle} 
has a pseudo-cross section, and hence has a natural structure as a principal 
twisted cartesian product.  It is not hard to show that the principal twisted 
cartesian product~\eqref{eq:T(N) bundle} is equal to the universal twisted 
cartesian product (see pages 88--89 of \cite{May-SOAT})
\[
WG\to \overline{W}G 
\]
defined in terms of the canonical twisting function $t$ on $\overline{W}G$ defined by 
\[
t_n\colon (\overline{W}G)_n\to G_{n-1},\quad t_n(g_{n-1},\ldots,g_0) = g_{n-1}.  
\]
We summarize the preceding discussion in the following lemma.  

\begin{lemma} 
\label{lem:univ ptcp}
Let $\cC$ be a category with finite limits and let $G$ be a group in $s\cC$.  Then there is a 
canonical principal twisted cartesian product $WG\to \overline{W}G$ with structure group $G$.  
Moreover $WG$ has a natural structure as a group in $s\cC$ containing $G$ as a subgroup.  
\end{lemma}      

The only statement in Lemma~\ref{lem:univ ptcp} that has not been discussed above is the 
statement regarding the group structure on $WG$; this is a simple consequence of the description 
of $WG$ in terms of the total simplicial object functor.  We refer to \cite{Roberts} for 
further discussion of this.  

Finally, we note that there is another useful perspective on the universal principal twisted cartesian product 
$\pi\colon WG\to \overline{W}G$; the map $\pi$ is equal to the canonical map $\Dec_0\overline{W}G\to \overline{W}G$, where 
$\Dec_0\colon s\cC\to s\cC$ is the d\'{e}calage functor.  In this description 
the pseudo-cross section appears as a certain monadic structure on the functor $\Dec_0$.     

Recall (see for example  \cite{Duskin Mem,Illusie, Ste1}, that $\Dec_0$ is the
functor which shifts degrees up by one so that if $X$ is a simplicial  object
in $\cC$ then $\Dec_0(X)_n = X_{n+1}$ with the first face and degeneracy  map
at each level forgotten or `stripped away'.  In other  words $\Dec_0$ is the
functor induced by restriction along  the functor $\sigma_0\colon \Delta\to
\Delta$,  where $\sigma_0$ is defined by  $\sigma_0([n]) = \sigma([0],[n])$, where
$\sigma\colon  \Delta\times \Delta\to \Delta$ is ordinal sum, i.e.\
$\sigma([m],[n]) = [m+n+1]$.   Observe that the first face map at every level
defines a  simplicial map $d_{\mathrm{first}}\colon \Dec_0X\to X$ for any simplicial
object $X$ in $\cC$ which in degree $n$ is given by  $d_{0}\colon X_{n+1}\to
X_n$.

\section{Geometric realization of simplicial principal bundles} 
\label{sec:geom_realzn}

In this section we show that fiberwise geometric realization of a large class of 
simplicial parametrized principal bundles gives parametrized principal bundles.  
We discuss sufficient conditions on a simplicial parametrized group $G$ to ensure 
that $G$ is good and $\overline{W}G$ is proper 
(see Definition~\ref{def:proper simplicial obj} below).  

Recall from Section~\ref{sec: princ bundles} that if $P(M,G)$ is a simplicial parametrized 
principal bundle, then after taking fiberwise geometric 
realizations there is a principal action of $|G|$ on $|P|$ with $|M|$ 
as the space of orbits.  To prove that $|\pi|\colon |P|\to |M|$ is the projection 
map in a parametrized principal bundle all that remains is to prove that 
$|\pi|$ admits local sections.

We will show that a sufficient condition for 
this is that (a) the group $G$ be fibrant in the sense of Definition~\ref{def:fibrant param gp} 
and that (b) $M$ satisfies a cofibrancy condition.  This latter condition is 
the parametrized analog of May's notion of {\em proper} simplicial space introduced in 
\cite{May-GILS}.  In fact this notion, and the allied notion of a 
{\em good} simplicial space \cite{Segal-Cats}, makes sense in any topological bicomplete 
category.  

\begin{definition} 
\label{def:proper simplicial obj}
	
	Let $\cC$ be a bicomplete topological category.  A  simplicial object $X$
	in $\cC$ is called {\em proper} if the latching maps $L_nX\to X_n$ are
	$\bar{h}$-cofibrations for all $n\geq 0$;  $X$ is called {\em good} if all
	of the degeneracy morphisms  $s_i\colon X_n\to X_{n+1}$ are 
	$\bar{h}$-cofibrations.

\end{definition}

In particular, specialized to the case where $\cC = \cKB$, we obtain the 
notion of a proper simplicial parametrized space.  With these definitions 
understood, we re-state Theorem~\ref{main result} from the Introduction.  

\mainresult*

Since the proof of Theorem~\ref{main result} is somewhat technical we have deferred it 
to Section~\ref{app: proof of main result}.  We discuss some consequences.  Observe that, subject 
to the hypotheses above, if $P\to M$ is a principal twisted cartesian product 
with structure group $G$, then $|P|\to |M|$ is a numerable parametrized principal 
$|G|$ bundle.  An example of special interest is the universal principal 
twisted cartesian product $WG\to \overline{W}G$ (Lemma~\ref{lem:univ ptcp}); 
in order to apply Theorem~\ref{main result} 
in this case we need to investigate sufficient conditions for $\overline{W}G$ to 
be proper.    

In principle, it is easier to check that a simplicial object is good than it is 
to check that it is proper.  In Appendix~\ref{app:good implies proper} we give a 
proof, in the setting of a topological bicomplete category, of 
Proposition~\ref{prop:good_implies_proper}, which says that every 
good simplicial object is proper.  This fact is standard for simplicial 
spaces (see for instance \cite{GaunceLewis}; we show that the proof given in 
op.\ cit.\ carries through to this more general setting).  Therefore, we search 
for a condition on the simplicial parametrized group $G$ which ensures that 
$\overline{W}G$ is proper.       

Recall (Definition~\ref{def:well-sectioned group}) the notion of a 
well-sectioned simplicial parametrized group.  We will say that a simplicial 
parametrized group $G$ is a {\em good} simplicial group if the object in $s\cKB$ underlying
$G$ is good.   We recall the statement of Proposition~\ref{well pointed simp grp} from 
the introduction; it gives a condition on $G$ which ensures that 
$G$ is good, and that $\overline{W}G$ is good and hence proper.  

\wellpointedprop*

We have deferred the proof of Proposition~\ref{well pointed simp grp} to 
Section~\ref{app:proof of propn 3}.  
Note that there is a partial converse to the first statement: if $G$ is a good simplicial group in
$\cKB$ such that $G_0$ is well-sectioned, then $G_n$ is well-sectioned for
every $n\geq 0$.

Combining Theorem~\ref{main result}, Proposition~\ref{well pointed simp grp} 
and Lemma~\ref{lem:univ ptcp}
we obtain Proposition~\ref{corr} from the Introduction.  

\WGcorr*

Now we turn to the statement and proof of the main result of this paper.  
Let $G$ denote a parametrized group.  In \cite{Crabb-James} (see pages 37--39) a construction   of a universal
parametrized principal $G$-bundle is given, based on the Milnor  construction of a
universal bundle, using infinite joins.  This  model of the universal bundle
is very useful as it makes almost no  assumptions on $G$.  We will 
impose a mild restriction  on $G$---we will require that $G$ is well-sectioned---and build a 
model with more convenient properties.

If $G$ is well-sectioned, 
Proposition~\ref{corr} specializes, with $G$ regarded as a constant 
simplicial parametrized group, to the statement that 
\[
|WG|\to |\overline{W}G| 
\] 
is a numerable parametrized principal $G$-bundle.  Here 
$\overline{W}G$ is the simplicial parametrized space 
whose $n$-simplices are described in~\eqref{eq:simp BG} and whose 
face and degeneracy maps  are described in~\eqref{eq:face maps for BG} 
and~\eqref{eq:degen maps for BG}.  In the remainder of this section we shall write 
\[
BG := |\overline{W}G|\quad \text{and}\quad EG:= |WG|,  
\]
since, as we will see, the parametrized $G$-bundle $EG\to BG$ is a model 
for the universal parametrized $G$-bundle.  Firstly, let us note that 
if $H$ is another parametrized group, then there is a canonical isomorphism 
$\overline{W}(G\times_B H) = \overline{W}G\times_B \overline{W}H$ and hence 
a canonical isomorphism $B(G\times H) = BG\times_B BH$, since the fiberwise 
geometric realization functor preserves finite limits.  Thus by construction the classifying space functor $B(-)$ is product-preserving.

Recall Theorem~\ref{classifying thy} from the Introduction.    

\classifyingthm*
    
We now turn to the proof of this theorem.    

\begin{proof}
We make use of the fact that $H^1(M,G)_{\cKB}$ is  isomorphic
to the set of fiberwise concordance classes of fiberwise principal  $G$
bundles on $M$ (Corollary~\ref{concordance}).   We define a map
\begin{gather} 
	[M,BG]_{\cKB} \to H^1(M,G)_{\cKB} \label{isom1} \\ 
	[f] \mapsto [f^*EG] \notag 
\end{gather} 
for $f\colon M\to BG$.  It is easy to verify that this map is well defined.
To prove that it is a bijection we construct an inverse.  For this we need
some preparation.  

Suppose that $P$ is a parametrized principal $G$-bundle on $M$.  
Recall that the {\em \v{C}ech nerve} $\check{C}(P)$ 
of $P\to M$ is the augmented simplicial object 
\begin{equation} 
\label{cech nerve} 
	\xymatrix{ 
		\cdots \ar@<2ex>[r]_(0.3){:} & \ar@<2ex>[l]_(0.7){:} P\times_{M} P\times_{M}P 
		\ar@<2ex>[r] \ar@<1ex>[r] 
		\ar[r] & P\times_{M}P 
		\ar@<1ex>[l] \ar@<2ex>[l] 
		\ar@<1ex>[r] \ar[r] & P \ar@<1ex>[l] \ar[r] & M 
	}
\end{equation}
in $\cK_{/B}$ where the face and degeneracy maps are given by omission and
inclusions by diagonals.  Since $\check{C}(P)$ is augmented over $M$
it follows on taking fiberwise geometric realizations that we obtain a map
\[
	|\check{C}(P)|\to M  
\]
in $\cK_{/B}$.  The \v{C}ech nerve $\check{C}(Y)$ can of course be  defined for
any map $\pi\colon Y\to M$. It is a well known fact (essentially going back
to \cite{Segal-IHES})  that if $\pi$ admits local sections and $M$ is
paracompact then the  map $|\check{C}(Y)|\to M$ is a homotopy equivalence.
The following Lemma is a straightforward variation on this result whose proof we leave 
to the reader.

\begin{lemma} 
\label{segal lemma}

	If $\pi\colon Y\to M$ is a map in $\cK_{/B}$ which admits local  sections and
	$M$ is paracompact then the canonical map
	\[
		|\check{C}(Y)|\to M 
	\] 
	is a fiberwise homotopy equivalence.  

\end{lemma} 

%
%

With these preparations out of the way, we can return to the  problem of
defining an inverse for the map $[M,BG]_{\cK_{/B}}\to H^1(M,G)_{\cKB}$. Let $\pi$ be the projection map $P \to M$ of the bundle.  Since $G$ acts
principally on $P$, there exist maps
\[ 
	P\times_M P \to G,\  P\times_M P\times_M P\to G\times G,\ \ldots \ \text{etc}
\] 
which fit together to give a simplicial map $\check{C}(P)\to  \overline{W}G$.  
Observe that there is another simplicial map $\check{C}(\pi^*P)\to WG$ defined in an 
analogous fashion which forms part of 
a pullback diagram 
\[
\begin{xy} 
	(-10,7.5)*+{\check{C}(\pi^*P)}="1"; 
	(-10,-7.5)*+{\check{C}(P)}="2";
	(10,7.5)*+{WG}="3";
	(10,-7.5)*+{\overline{W}G}="4"; 
	{\ar "1";"2"};
	{\ar "1";"3"}; 
	{\ar "2";"4"};
	{\ar "3";"4"};
\end{xy}
\]
in $s\cKB$.  
On taking fiberwise geometric realizations we obtain a map
\begin{equation} 
\label{class map1}
	|\check{C}(P)|\to |\overline{W}G| =: BG
\end{equation} 
in $\cK_{/B}$.  Let $\sigma\colon M\to  |\check{C}(P)|$ denote a homotopy inverse to the map  
$|\check{C}(P)|\to M$ from
Lemma~\ref{segal lemma}.  Composing $\sigma$ with the map~\eqref{class map1}  
gives a map $M\to BG$.  It is clear that this map
respects the relation of concordance (recall that we are identifying
$H^1(M,G)_{\cKB}$  with the set of fiberwise concordance classes using Corollary~\ref{concordance}) to give a map
\begin{equation} 
\label{isom2}
	H^1(M,G)_{\cKB}\to [M,BG]_{\cK_{/B}}.  
\end{equation}
We need to prove that this map is the inverse of the  map~\eqref{isom1}.  We
first examine the composite $H^1(M,G)_{\cKB}\to  [M,BG]_{\cK_{/B}} \to H^1(M,G)_{\cKB}$.  To
show that  this is the identity we need to show that the pullback of $EG\to
BG$ under the map $|\check{C}(P)|\to BG$ (equation~\eqref{class map1} above) is fiberwise isomorphic to $q^*P$
where we use $q$ to  denote
the map $|\check{C}(P)|\to M$.  For it then follows that  the pullback of $EG\to
BG$ under the composite map $M\to  |\check{C}(P)|\to BG$ is isomorphic to
$\sigma^*q^*P\cong P$. Observe that on taking fiberwise geometric 
realizations we obtain the commutative diagram
\[ 
	\xymatrix{ 
		P \ar[d] & |\check{C}(\pi^*P)| \ar[d] \ar[l] \ar[r] & |WG| \ar[d] \\ 
		M & |\check{C}(P)| \ar[l] \ar[r] & |\overline{W}G| 
	} 
\] 
in $\cK_{/B}$ in which each square is a pullback. 
Hence it follows  that
the composite $H^1(M,G)_{\cKB}\to [M,BG]_{\cK_{/B}} \to H^1(M,G)_{\cKB}$ is the identity.

Now we examine the composite map $[M,BG]_{\cK_{/B}}\to H^1(M,G)_{\cKB}\to
[M,BG]_{\cK_{/B}}$.  To prove that this is the identity it is sufficient to prove
the following: in the diagram
\begin{equation} 
\label{fib hom} 
	\begin{xy} 
		(0,7.5)*+{|\check{C}(EG)|}="1"; 
		(0,-7.5)*+{BG}="2"; 
		(30,7.5)*+{|\overline{W}G| = BG}="3"; 
		{\ar "1";"2"}; 
		{\ar "1";"3"};
	\end{xy} 
\end{equation} 
the two maps $|\check{C}(EG)|\to BG$ are fiberwise homotopic.  Here the horizontal map is the fiberwise geometric 
realization of the map~\eqref{class map1} (in the special case of $P = EG$) and the vertical map is the canonical map 
obtained by the augmentation of the \v{C}ech nerve $\check{C}(EG)$ of $EG\to BG$.

The existence of this fiberwise homotopy can be understood 
as a simple fact about the {\em total d\'{e}calage} functor; 
therefore we shall need a short interlude to discuss this latter object.  
Recall (see for example \cite{Duskin Mem, Ste1}) that $\Dec\colon s\cKB\to ss\cKB$ 
is the functor induced by restriction along the ordinal sum functor $\sigma\colon \Delta\times 
\Delta\to \Delta$ defined above.  Thus 
$\Dec\, X$ is the bisimplicial parametrized space whose columns form the simplicial object 
\[
\xymatrix{ 
\Dec_0X \ar@<-1ex>[r] & \ar[l] \ar@<-1ex>[l] \Dec_1 X \ar@<-1ex>[r] \ar@<-2ex>[r] & 
\ar[l] \ar@<-1ex>[l] \ar@<-2ex>[l] \Dec_2 X  \cdots}
\]  
where $\Dec_nX = (\Dec_0)^{n+1}X$.  
In particular the 0-skeleton of $\Dec\, X$ is $\Dec_0 X$.   
Ordinal sum with the empty set defines canonical natural transformations $p_1\to \sigma$ 
and $p_2\to \sigma$, where $p_1,p_2\colon \Delta\times \Delta\to \Delta$ denote 
the projections onto the first and second factors.  Hence the total d\'{e}calage $\Dec\, X$ 
of a simplicial parametrized space $X$ comes equipped with row and column augmentations 
$\Dec\, X\to p_1^*X$ and $\Dec\, X\to p_2^*X$ respectively.  On taking 
diagonals and fiberwise geometric realizations, we obtain a diagram 
\begin{equation}
\label{eq:can diagram for Dec}
\begin{xy} 
		(0,7.5)*+{|d\Dec\, X|}="1"; 
		(0,-7.5)*+{X}="2"; 
		(20,7.5)*+{X}="3"; 
		{\ar_-{q_1} "1";"2"}; 
		{\ar^-{q_2} "1";"3"};
	\end{xy} 
\end{equation}
A further useful property of the total d\'{e}calage is that the fiberwise geometric realization 
$|d\Dec\, X|$ is isomorphic to $X$, for $d\Dec\, X$ is easily seen to be equal to the edge-wise subdivision of 
$X$ as defined in \cite{BokHsiMad}.

We return to the problem at hand.  
Recall, see the remarks following Lemma~\ref{lem:univ ptcp}, that $WG = \Dec_0\overline{W}G$.  
It follows, by an adjointness argument using the fact that $\sk_0 \Dec\, \overline{W}G = WG$, 
that there is a canonical map of bisimplicial parametrized spaces 
\begin{equation}
\label{eq: canonical map from Dec}
\Dec\,  \overline{W}G\to \check{C}(WG) = \cosk_0(WG).  
\end{equation}
This map is easily checked to be an isomorphism and moreover the diagram~\eqref{eq:can diagram for Dec} 
is equal to the diagram~\eqref{fib hom} above with $X = \check{C}(WG)$.  Note that we also obtain the 
not-so-obvious fact that $|\check{C}(EG)|$ is isomorphic to $BG$.  

Thus to prove that the two maps in~\eqref{fib hom} are fiberwise homotopic, it suffices 
to prove that the two maps in~\eqref{eq:can diagram for Dec} are fiberwise homotopic.       
We shall prove that if $X$ is a simplicial parametrized space, then there is a canonical simplicial  
homotopy $d\Dec\, X\otimes \Delta[1]\to X$ from $q_1$ to $q_2$.  
Taking fiberwise geometric realizations then gives the required fiberwise homotopy.  

By adjointness, exhibiting such a simplicial homotopy, is equivalent to exhibiting a map 
$d\Dec\, X\to X^{\Delta[1]}$ 
such that the diagram 
\begin{equation} 
\label{eq:can fiber hty diag}
\begin{xy} 
		(-10,0)*+{d\Dec\, X}="1"; 
		(10,10)*+{X}="2";
		(10,-10)*+{X}="3"; 
		(10,0)*+{X^{\Delta[1]}}="4"; 
		{\ar "1";"4"};
		{\ar_-{q_2} "1";"3"}; 
		{\ar^-{q_1} "1";"2"};
		{\ar "4";"2"};
		{\ar "4";"3"};
	\end{xy} 
\end{equation}
commutes, where $X^{\Delta[1]}$ denotes the usual simplicial path space of $X$, and the two 
projections $X^{\Delta[1]}\to X$ are induced by the inclusions $0,1\colon \Delta[0]\to \Delta[1]$.  
To be more concrete, $X^{\Delta[1]}$ is the simplicial parametrized space whose 
space of $n$-simplices is the 
generalized matching
 object
 \[
 	(X^{\Delta[1]})_n = M_{\Delta[n]\times \Delta[1]}X 
 \] 
 (see VII 1.21 of \cite{GJ}).  It is
 an easy calculation to see that the object of $n$-simplices of
 $d\Dec(X)$ is given by
 \[
 	(d\Dec(X))_n = M_{\Delta[n]\star \Delta[n]}X = M_{\Delta[2n+1]}X, 
 \]
 where $\Delta[n]\star \Delta[n]$ denotes the {\em join} of $\Delta[n]$ with itself~\cite{Ehlers-Porter}. 
 To construct the map $d\Dec\, X\to X^{\Delta[1]}$ it suffices 
 to construct a simplicial map  
 \[
 	\Delta[n]\times \Delta[1]\to \Delta[2n+1], 
 \] 
 natural in $[n]$, such that the diagram 
 \[
 	\xymatrix{ 
		& \Delta[n] \ar[d] \ar[dl]!UR_-{\tilde q_1} \\ 
		\Delta[n]\star \Delta[n] = \Delta[2n+1] &  \ar[l] \Delta[n]\times \Delta[1] \\ 
		& \Delta[n] \ar[u] \ar[ul]!DR^-{\tilde q_2} \\
		}
\]
 commutes, where the two maps $\tilde q_1,\tilde q_2\colon \Delta[n]\to \Delta[2n+1]$ are induced by
 $\sigma([n], \emptyset)\to \sigma([n], [n])$ and $\sigma(\emptyset, [n])\to \sigma([n],
 [n])$  respectively.  The required homotopy is the nerve of the canonical natural 
 transformation $\alpha\colon \tilde q_1\to \tilde q_2$ defined by $\alpha(i)\colon i\to n+i+1$.  It 
 is easy to see that this map is natural in $n$ in the appropriate sense.       
This finishes the proof of Theorem~\ref{classifying thy}.
\end{proof}

\section{Proof of Theorem~\ref{main result}} 
\label{app: proof of main result}

In this Section, we prove Theorem~\ref{main result}.  First recall the statement of 
this theorem.  

\mainresult*

The proof of Theorem~\ref{main result} is a variation on the approach of  the papers
\cite{May-Classifying, McCord, Steenrod} (which deal with the  case where $G$
is a constant simplicial group) to the case  where $G$ is an arbitrary group
in $s\cKB$.   We note that an important ingredient in \cite{May-Classifying, McCord,
Steenrod} is the notion of an  equivariant NDR pair, a notion which we have
already explained  (see Section~\ref{sec:param_gp} above)  has a
straightforward generalization to the parametrized setting.

\begin{proof}[Proof of Theorem~\ref{main result}]

Let $n\geq 0$ be an integer.   Recall that the $n^{\text{th}}$ skeleton $\sk_n
M$ of $M$ comes  equipped with a map $\sk_n M\to M$ and that there are natural
maps $\sk_n M\to \sk_m M$ whenever $m\leq n$.   Recall also that $M = \colim_n
\sk_n M$ and that there is a pushout diagram of the form 
\begin{equation} 
\label{eq:skeletal pushout}
	\begin{xy}
		(-25,7.5)*+{(M_n\otimes \partial \Delta[n])\cup (L_n M\otimes \Delta[n])}="1";  
		(-25,-7.5)*+{M_n\otimes \Delta[n]}="2";
		(25,7.5)*+{\sk_{n-1} M}="3"; 
		(25,-7.5)*+{ \sk_n M}="4";
		{\ar "1";"3"}; 
		{\ar "1";"2"};
		{\ar "2";"4"}; 
		{\ar "3";"4"};
	\end{xy}
\end{equation}
(see for instance Proposition VII 1.7 of \cite{GJ}), where $\Delta[n]$ denotes the simplicial 
$n$-simplex and $\partial\Delta[n]$ denotes its boundary.

We use the $n^{\text{th}}$ skeletons of $M$ to define a filtration
$|P|_0\subset |P|_1 \subset \cdots \subset |P|_n\subset \cdots \subset |P|$ of
$|P|$ as follows.   The canonical maps $\sk_n M\to M$ induce by pullback   simplicial
principal bundles with structure group $G$ on each of the simplicial spaces
$\sk_n M$.   Let $|P|_n =  |\sk_n M\times_M P|$.  Observe that $|P|_n\subset |P|_{n+1}$ 
and $|P|_n\subset |P|$ are closed inclusions for all $n\geq 0$.  For convenience of notation
we will also denote $|\sk_n M|$ by $|M|_n$, but note the potential  confusion with 
$|P|_n$: we remind the reader that this does \emph{not} denote the geometric realization of 
the $n$-skeleton of $P$.  Recall that $M = \colim_n \sk_n M$ and hence $|M| =  \colim_n
|M|_n$ in $\cKB$.  We claim that $P = \colim_n (\sk_nM\times_M P)$.  This is
easy to see in the special case that $P$ is trivial.  We can reduce the
general statement  to this special case, since $P$ is a colimit of trivial
bundles and colimits commute amongst themselves.

The map $|P|\to |M|$ is a quotient map, since the map $\sqcup_{n\geq 0}
P_n\times \Delta^n\to  \sqcup_{n\geq 0} M_n\times \Delta^n$ is a quotient
map, and both of the maps $\sqcup_{n\geq 0} P_n\times  \Delta^n\to |P|$ and
$\sqcup_{n\geq 0} M_n\times \Delta^n\to |M|$ are quotient maps.  Since the
diagram
\[
	\xymatrix{ 
		|P|_n \ar[r] \ar[d] & |P| \ar[d] \\ 
		|M|_n \ar[r] & |M| 
	} 
\]
is a pullback, we see that $|P|_n\to |M|_n$ is also a quotient map
($|M|_n\to |M|$ is a closed inclusion, and quotient maps pullback along closed
inclusions to  quotient maps).   In particular $|M|_n$ has the quotient
topology induced by the map $|\pi|\colon |P|_n\to |M|_n$.

The main step in our proof is to prove that $(|P|_n,|P|_{n-1})$ is a
$|G|$-fiberwise NDR pair in $\cKB$ for all $n\geq 1$, so that we  can apply
the method of \cite{May-Classifying, McCord, Steenrod}.  As a first step in
this direction we have the following lemma.

\begin{lemma}  
\label{pushout lemma for realization}

	For every $n\geq 1$ we have a pushout diagram in $s\cKB$ of the form

	\begin{equation}
	\label{first pushout}
		\begin{xy}
		(-25,7.5)*+{((M_n\otimes \partial \Delta[n])\cup (L_n M\otimes \Delta[n]))\times_M P}="1";  
		(-25,-7.5)*+{(M_n\otimes \Delta[n])\times_M P}="2";
		(25,7.5)*+{\sk_{n-1} M\times_M  P}="3"; 
		(25,-7.5)*+{ \sk_n M\times_M P}="4";
		{\ar "1";"3"}; 
		{\ar "1";"2"};
		{\ar "2";"4"}; 
		{\ar "3";"4"};
		\end{xy}
	\end{equation}
	
\end{lemma}

\begin{proof}

Observe that the canonical map from the pushout to $\sk_nM\times_M P$ is a
continuous bijection  in each degree.  Therefore it suffices to show that for
each $m\geq 0$ the induced map
\begin{equation}
\label{coproduct of bundles}
	((M_n\otimes \Delta[n])_m\times_{M_m}P_m)\sqcup ((\sk_{n-1}M)_m\times_{M_m}P_m)\to 
	(\sk_nM)_m\times_{M_m}P_m 
\end{equation}
is a quotient map.  The map~\eqref{coproduct of bundles} is the map of
fiberwise principal bundles  induced by pullback along the quotient map
\[
	(M_n\otimes \Delta[n])_m\sqcup (\sk_{n-1}M)_m\to (\sk_nM)_m.  
\]
Therefore to prove the lemma it suffices to establish the  following claim: if
$P\to M$ is a fiberwise principal bundle and  $f\colon N\to M$ is a quotient
map in $\cKB$, then  $f^*P\to P$ is also a quotient map.   To see this
observe that since $P$ can be constructed  as a quotient of a coproduct of
spaces of the  form $U\times_B G$, and $f^*P$ can be constructed as a
quotient of a coproduct of spaces of the form $f^{-1}U\times_B G$, it suffices
to  prove that $f^{-1}U\times_B G\to U\times_B G$ is a quotient  map for any
open set $U\subset M$.  Since the functor $(-)\times_B G$ preserves  colimits
this follows from the fact that $f^{-1}U\to U$ is a quotient map, since
$U\subset M$ is open.
\end{proof} 

Continuing the proof of Theorem~\ref{main result}, the second step is to show that in the 
diagram~\eqref{first pushout} the realization of the left hand vertical map is an $\bar{f}$-cofibration 
in $|G|\cKB$.  For this we will need the hypotheses that 
each $P_n\to M_n$ is a numerable principal $G_n$ bundle, and that $M$ is proper.  

\begin{lemma}
\label{lem:second lemma}

	For every $n\geq 1$, the map 
	\begin{equation}
	\label{eq: pullback of inclusion}
		|(M_n\otimes \partial\Delta[n])\cup (L_nM\otimes \Delta[n])\times_M P|\to 
		|M_n\otimes \Delta[n]\times_M P| 
	\end{equation}
	is an $\bar{f}$-cofibration in $|G|\cKB$ and 
	hence $(|P|_n,|P|_{n-1})$ is
	a $|G|$-fiberwise NDR  pair in $\cKB$ for all $n\geq 1$.

\end{lemma} 

\begin{proof}
Using the fact that geometric realization commutes with pullbacks, we obtain a pullback diagram 
\[
	\xymatrix{ 
		|((M_n\otimes \partial \Delta[n])\cup (L_nM\otimes \Delta[n]))\times_M P| \ar[d] \ar[r] & 
		|M_n\otimes \Delta[n]\times_M P| \ar[d] \\ 
		(M_n\times \partial\Delta^n)\cup (L_nM\times \Delta^n) \ar[r] & 
		M_n\times \Delta^n.
	}
\]
Since $M$ is proper, the closed inclusion $L_nM\subset M_n$ is an
$\bar{f}$-cofibration and  standard results show that this induces 
a closed inclusion   
$(M_n\times \partial\Delta^n)\cup (L_nM\times
\Delta^n)\to M_n\times \Delta^n$ which is also an $\bar{f}$-cofibration.  Therefore 
if we can show that
\[
	|M_n\otimes \Delta[n]\times_M P| \to M_n\times \Delta^n 
\]
is a numerable fiberwise principal $|G|$ bundle in $\cKB$, then we may use
Proposition~\ref{Stroms theorem} to deduce that the closed
inclusion~\eqref{eq: pullback of inclusion}  is an $\bar{f}$-cofibration in
$\cKB$.  It then follows from Lemma~\ref{pushout lemma for realization}  that
$|P|_{n-1}\to |P|_n$ is an $\bar{f}$-cofibration, since these are  preserved
under pushout.  Finally, it follows from Lemma~\ref{ndr} that
$(|P|_n,|P|_{n-1})$ is a $|G|$-fiberwise NDR pair.

Since we have shown that $|P| \to |M|$ satisfies the condition (ii) of
Definition~\ref{def:fiberwise bundle} for the group $|G|$, and this
condition is stable under pullback, $|M_n\otimes \Delta[n]\times_M P|
\to M_n\times \Delta^n$ also satisfies the condition (ii). We thus only
need to show that this map satisfies the condition (i). That is, it admits
local sections relative to a numerable open cover of $M_n\times \Delta^n$.  For this, consider the
commutative diagram
\[
	\xymatrix{ 
		P_n\times \Delta^n \ar[r] \ar[d] & |P| \ar[d] \\ 
		M_n\times \Delta^n \ar[r] & |M| 
	} 
\]
where the horizontal maps are the canonical ones into the colimits defining
$|P|$ and $|M|$.  The map  $P_n\times \Delta^n\to |P|$ factors through
$|M_n\otimes \Delta[n]\times_M P|$ and hence $|M_n\otimes \Delta[n]\times_M
P|\to M_n\times \Delta^n$  admits local sections relative to a numerable open
cover of $M_n\times \Delta^n$ since the principal $G_n$ bundle $P_n\times
\Delta^n\to M_n\times \Delta^n$ does by  hypothesis.
\end{proof}

We now proceed in our proof of Theorem~\ref{main result} in analogy with the arguments in \cite{May-Classifying, McCord,
Steenrod}.   Since $(|P|_n,|P|_{n-1})$ is a fiberwise $|G|$-equivariant NDR
pair for every $n\geq 1$ and $|P| = \colim_n |P|_n$, we see  that
$(|P|,|P|_n)$ is a fiberwise $|G|$-equivariant NDR pair for every $n\geq 0$ (by Lemma~\ref{ndr} and Lemma~\ref{seq colimit lemma}).   For any $n\geq 0$ let $h_n\colon
|P|\times I\to |P|$ and  $u_n\colon |P|\to I$ be a representation of
$(|P|,|P|_n)$ as a fiberwise $|G|$-equivariant NDR pair.   Define functions
$\hat{\rho}_n\colon |P|\to I$ for every $n\geq 1$ by
\[ 
	\hat{\rho}_n(x) = (1 - u_n(x))u_{n-1}(h_n(x,1)).   
\] 

The functions $\hat{\rho}_n$ are easily seen to be $|G|$-invariant and hence
descend to functions $\rho_n\colon |M|\to I$.   Let $U_n =
\hat{\rho}_n^{-1}(0,1]$ and let $V_n = \rho_n^{-1}(0,1]$ so that  $U_n =
|\pi|^{-1}V_n$ (and hence $U_n$ is $|G|$-invariant).       Following \cite
{May-Classifying, McCord} let $r_n\colon |P|\to |P|$ denote the map
$r_n(|x,t|) = h_n(|x,t|,1)$.   Then we have (see \cite{May-Classifying,
McCord}) the following chain of inclusions
\[ 
	|P|_n \setminus |P|_{n-1} \subset U_n \subset r_n^{-1}(|P|_n \setminus |P|_{n-1}) 
\] 
Observe that we have a commutative diagram 
\begin{equation}
\label{comm diagram}
	\begin{xy} 
	(-25,7.5)*+{U_n}="1";
	(0,7.5)*+{|P|_n\setminus|P|_{n-1}}="2";
	(35,7.5)*+{|M_n\otimes \Delta[n]\times_M P|}="3"; 
	(-25,-7.5)*+{V_n}="4"; 
	(0,-7.5)*+{|M|_n\setminus|M|_{n-1}}="5"; 
	(35,-7.5)*+{|M_n\otimes \Delta[n]|}="6";
	{\ar_-{|\pi|} "1";"4"};
	{\ar^-{r_n} "1";"2"};
	{\ar "2";"3"}; 
	{\ar "3";"6"}; 
	{\ar "2";"5"};
	{\ar "4";"5"}; 
	{\ar "5";"6"};
	\end{xy}
\end{equation}
in which the top horizontal maps are $|G|$-equivariant.   The lower right hand map in 
this diagram arises as follows: after taking geometric realizations in~\eqref{eq:skeletal pushout}, 
we see that there is an isomorphism 
\[
|M|_n\setminus |M|_{n-1} = |M_n\otimes \Delta[n]|\setminus 
|(M_n\otimes \partial \Delta[n])\cup (L_n M\otimes \Delta[n])| 
\]
and hence a natural inclusion $|M|_n\setminus |M|_{n-1}\subset |M_n\otimes \Delta[n]|$.  

After the previous
Lemma~\ref{lem:second lemma} we observed that  $|M_n\otimes \Delta[n]\times_M
P|\to |M_n\otimes \Delta[n]|$ is a numerable  fiberwise principal $|G|$-bundle
in $\cKB$ and hence is locally trivial.  Using  local sections of this map,
we can find an open cover $(V_{n,i})$ of $V_n$ and  
$|G|$-equivariant maps $\zeta_{n,i}\colon U_{n,i}\to |G|$, where $U_{n,i} = |\pi|^{-1}V_{n,i}$.
Then we can define $|G|$-invariant maps  $\hat{\sigma}_{n,i}\colon U_{n,i}\to
U_{n,i}$ by $\hat{\sigma}_{n,i}(x) = x\zeta_{n,i}(x)^{-1}$.   Since
$\hat{\sigma}_{n,i}$ is  $|G|$-invariant, it descends to define a unique map
$\sigma_{n,i}\colon V_{n,i}\to U_{n,i}$ so that the diagram
\[ 
	\xymatrix{ 
		U_{n,i} \ar[d]_-{|\pi|} \ar[r]^-{\hat{\sigma}_{n,i}} & U_{n,i} \\
		V_{n,i} \ar[ur]_-{\sigma_{n,i}} 
	} 
\] 
commutes. The set $V_{n,i}$ has the quotient topology induced by $|\pi|$  and hence
$\sigma_{n,i}$ is continuous.  Clearly $\sigma_{n,i}$ is a section of $|\pi|$.
Thus we have proven that there exist trivializations of $|\pi|\colon |P|\to
|M|$ over the open subsets $V_{n,i}$.

It remains to prove the statement regarding the numerability of the bundle 
$|P|\to |M|$.  We argue as follows.  
From the proof above we obtain the commutative diagram~\eqref{comm diagram}.
In this case the bundle $|M_n\otimes \Delta[n]\times_M P|\to |M_n\otimes
\Delta[n]|$ is  trivial, and therefore we can define $|G|$-equivariant maps
$\zeta_n\colon U_n\to |G|$.   In exactly the same way as above we can use the
maps $\zeta_n$ to define  $|G|$-invariant maps $\hat{\sigma}_n\colon U_n\to
U_n$ which descend to sections  $\sigma_n\colon V_n\to U_n$ of $|\pi|$.  The
problem now is to show that  the open cover $(V_n)$ is numerable. To do this
we use the functions $\rho_n\colon V_n\to I$ constructed earlier.    The
collection of functions  $(\rho_n)$ may not be locally finite, this can be
fixed however using the method of Dold \cite[Proof of Proposition 6.7]{Dold}; one defines new  functions
$\phi_n\colon U_n\to I$ with $\supp (\phi_n)\subset U_n$ by
\[ 
	\phi_n(x) = \max\left( 0,\rho_n(x) - n\sum^{n-1}_{i=1}\rho_i(x)\right). 
\]
Then one can check as in \cite{Dold} that the collection of functions
$(\phi_n)$ is locally finite.  It is now  clear how to form a partition of
unity from the $\phi_n$. 
This ends the proof of Theorem~\ref{main result}.
\end{proof}

\section{Proof of Proposition~\ref{well pointed simp grp}} 
\label{app:proof of propn 3}

Recall the statement of Proposition~\ref{well pointed simp grp}.  

\wellpointedprop*

\begin{proof}
We prove statement (1).  We need to show that $s_i\colon G_n\to G_{n+1}$ is an  $\bar{f}$-cofibration
for all $0\leq i\leq n$ and all $n\geq 0$.   Since $s_i$ is a section of the
corresponding  face operator $d_i$, we can identify $s_i$ with the map $G_n\to
G_n\times_B \ker(d_i)$  which sends $g\mapsto (g,1)$.  Therefore, by Lemma
\ref{products of cofibrations} below, to prove that $s_i$ is an
$\bar{f}$-cofibration  it is sufficient to prove that  $\ker(d_i)$ is well
sectioned.
For this, we observe that $\ker(d_i)$  is a retract of $G_{n+1}$ by the map
$G_{n+1}\to \ker(d_i)$  sending $g$ to $gs_id_i(g)^{-1}$.  Therefore the
section $B\to \ker(d_i)$ is an $\bar{f}$-cofibration since it is a  retract of
the map $B\to G_{n+1}$ which is an $\bar{f}$-cofibration by hypothesis.

We prove statement (2).  From what we have just proved, we have that each degeneracy map of $G$ is an 
$\bar{f}$-cofibration.  Lemma~\ref{products of cofibrations} below implies that 
the degeneracies of $\overline{W}G$ are $\bar{f}$-cofibrations and hence 
Proposition~\ref{param lillig} in Appendix~\ref{app:good implies proper} implies that $\overline{W}G$ is proper.

Finally we prove statement (3).  Since $G$ is well-sectioned, the simplicial object $G$ is proper, and hence
the inclusion $|G|_n\subset |G|_{n+1}$ 
is an $\bar{f}$-cofibration for all
$n\geq 0$ (with the notation of the proof of Theorem~\ref{main result}).  
This follows from the fact that $|G|_n\subset |G|_{n+1}$ is a pushout 
of $|G_n\otimes \partial\Delta[n])\cup (L_n G\otimes \Delta[n])|\to |G_n\otimes \Delta[n]|$, 
which is an $\bar{f}$-cofibration using Proposition~\ref{well pointed simp grp} 
and the fact that $\cKB$ is a topological model category.  
Therefore the inclusion $|G|_n\subset |G|$ is an 
$\bar{f}$-cofibration for
all $n\geq 0$ (by the non-equivariant version of Lemma~\ref{seq colimit lemma}).  
Since $|G|_0$ is well-sectioned and the composite of two
$\bar{f}$-cofibrations is an $\bar{f}$-cofibration, it follows that $|G|$ is
well-sectioned.    
\end{proof}

To complete the proof of Proposition~\ref{well pointed simp grp} we need to give the proof of the following lemma.  

\begin{lemma}
\label{products of cofibrations}

	Suppose that $A_1\to X$ and $A_2\to Y$ are $\bar{f}$-cofibrations in
	$\cKB$.  Then  $A_1\times_B A_2\to X\times_B Y$ is also an
	$\bar{f}$-cofibration.

\end{lemma}

\begin{proof}
It is clearly sufficient to prove that if  $A\to X$ is an
$\bar{f}$-cofibration and $Y$ is any space over $B$, then  $A\times_B Y \to
X\times_B Y$ is an $\bar{f}$-cofibration, in other words it has  the LLP with
respect to all $f$-acyclic $f$-fibrations $U\to V$.  By adjointness,  this is
equivalent to checking that $A\to X$ has the LLP against all maps of  the form
$\Map_B(Y,U)\to \Map_B(Y,V)$ where $U\to V$ is an $f$-acyclic $f$-fibration.

By an adjointness argument, the functor $\Map_B(Y,-)\colon \cKB\to \cKB$
preserves  $f$-fibrations.  It also preserves fiberwise homotopies: if
$g_0,g_1\colon X\to Z$ are  fiberwise homotopic through a fiberwise homotopy
$h\colon X\times I\to Z$, then  the maps $\Map_B(Y,g_0)$ and $\Map_B(Y,g_1)$
are fiberwise homotopic through  the fiberwise homotopy $\tilde{h}\colon
\Map_B(Y,X)\times I\to \Map_B(Y,Z)$ defined as  the composite
\begin{equation}
\label{Map_B fiberwise homotopy}
	\Map_B(Y,X)\times I\to \Map_B(Y,X\times I)\xrightarrow{\Map_B(Y,h)}
	\Map_B(Y,Z), 
\end{equation}
where the first map is the adjoint of the canonical map  $Y\times_B
\Map_B(Y,X)\times I\to X\times I$.  One can check that  $\tilde{h}$ so
defined   does give such a fiberwise homotopy as claimed.  It follows that the
functor $\Map_B(Y,-)$  preserves $f$-equivalences, and hence $f$-acyclic
$f$-fibrations, which proves the lemma.
\end{proof}

\begin{appendices}

\section{Good implies proper} 
\label{app:good implies proper}

Our goal in this section is to prove that a good simplicial object $X$  in a topological bicomplete 
category $\cC$ is automatically proper, provided that a generalization of Lillig's union 
theorem on cofibrations \cite{Lillig} holds in $\cC$, and an assumption on colimits 
in the slice categories $\cC_{/X_n}$ is met.  We begin by making the following definition.  

\begin{definition}
\label{lillig}

	Let $\cC$ be a topological bicomplete category.  We say that $\cC$
	satisfies the  {\em Lillig condition} if the following is true: 
	Given a pullback diagram in $\cC$,
	\[
		\xymatrix{ 
			A_3\ar[r] \ar[d] & A_2 \ar[d] \\ 
			A_1 \ar[r] & X 
		} 
	\] 
	such that the morphisms $A_1\to X$, $A_3\to X$ and $A_2\to X$ are
	$\bar{h}$-cofibrations, then the canonical map $A_1\cup_{A_3} A_2\to X$
	is an $\bar{h}$-cofibration.

\end{definition}

When $\cC = \cK$ this is Lillig's union theorem \cite{Lillig}.
We will prove shortly  that a reworking of the proof in \cite{Lillig} shows
that the Lillig condition holds  when $\cC = \cKB$; we do not know if this
condition holds more generally.

With this definition understood we can turn to our main goal in this appendix, which is 
the proof of the following proposition.  

\begin{proposition} \label{prop:good_implies_proper}

	Let $\cC$ be a topological bicomplete category and let $X$ be a 
	good simplicial object in $\cC$.  Suppose that the following two conditions are satisfied: 
	\begin{enumerate}
	\item $\cC$ satisfies the Lillig condition of Definition~\ref{lillig}, 

	\item $s_k\colon X_n\to X_{n+1}$ is properly extensive for all $n\geq 0 $ and all $0\leq k \leq n$.  
	\end{enumerate}
	Then $X$ is proper.

\end{proposition} 

Here we say that a map $f\colon X\to Y$ in $\cC$ is {\em properly extensive} 
if the pullback functors $f^*\colon \cC_{/Y}\to \cC_{/X}$ commutes with finite colimits.    
The proof of Proposition~\ref{prop:good_implies_proper} that we shall give is 
based on the proof of Corollary 2.4 (b) of \cite{GaunceLewis}.  We begin with 
some preparation.  

Recall (Definition~\ref{def:proper simplicial obj}) that a proper simplicial object $X$ in a topological bicomplete
category $\cC$ is one for  which the latching maps $L_nX\to X_n$ are
$\bar{h}$-cofibrations for all $n\geq 0$.        We need to examine the notion
of latching object in a little more detail.   
Recall (see for example Remark VII 1.8 of \cite{GJ}), that $L_nX$ may also be described as the coequalizer 
\begin{equation}
\label{coeq latching}
	\bigsqcup_{0\leq i<j\leq n-2} X_{n-2} \rightrightarrows \bigsqcup_{0\leq l\leq n-1} X_{n-1} \to L_nX 
\end{equation}
where the two maps defining the coequalizer arise from the simplicial identity
$s_is_{j-1} = s_{j}s_i$ if $i<j$ (see for example V Lemma 1.1 and VII Remark
1.8 of \cite{GJ}). 
It is well known that $L_0X = \emptyset$, $L_1X = X_0$ and
$L_2X = X_1\cup_{X_0} X_1$. 

It will be convenient to introduce a family of {\em partial latching objects} 
$L_{n,k}X$ associated to the simplicial object $X$ for $k=0,1,\ldots,n$.  
For $0\leq k\leq n$ we define $L_{n,k}X$ by the coequalizer 
\[
\bigsqcup_{0\leq i<j\leq k-1} X_{n-2}\rightrightarrows 
\bigsqcup^{k-1}_{l=0}X_{n-2} \to L_{n,k}X 
\]
where the restrictions of the two displayed maps to the summand labelled by the pair
$(i,j)$ are given by the composites 
\begin{align*} 
& X_{n-2} \xrightarrow{s_i} X_{n-1} \xrightarrow{\mathrm{in}_j} \bigsqcup^{k-1}_{l=0}X_{n-1} \\ 
& X_{n-2} \xrightarrow{s_j} X_{n-1} \xrightarrow{\mathrm{in}_i}\bigsqcup^{k-1}_{l=0}X_{n-1} 
\end{align*}
and where $\mathrm{in}_i,\mathrm{in}_j$ denote the inclusions into the summands labelled by 
$i$ and $j$.  Note that there are isomorphisms $L_{n,0}X\simeq \emptyset$, $L_{n,n}X\simeq L_nX$.  
Note also that there is a canonical map $L_{n,k}X\to X_n$ induced by the degeneracies 
$s_i\colon X_{n-1}\to X_n$ for $0\leq i\leq k-1$.    
These partial latching objects are precisely the objects $L_{n,k}X$ defined on pages 
362--363 of \cite{GJ}.  We have the following result: 

\begin{lemma}[\cite{GJ}, chapter VII Proposition 1.27] 
Let $X$ be a simplicial object in $\cC$.  Then for any $0\leq k\leq n-1$ there is a pushout diagram 
\[
\xymatrix{ 
L_{n-1,k}X \ar[r] \ar[d] & X_{n-1} \ar[d] \\ 
L_{n,k}X \ar[r] & L_{n,k+1}X } 
\]
\end{lemma} 

\begin{proof} 
The lemma follows from the statements (i)--(iii) below, together with the fact that colimits commute amongst themselves.  

\bigskip 

\noindent
(i) the diagram 
\[
\xymatrix{ 
\displaystyle{\bigsqcup_{0\leq i<j\leq k-1}} X_{n-3} \ar[r] \ar[d]_-{s_k} & 
\displaystyle{\bigsqcup_{0\leq i<j\leq k-1}}X_{n-3}\sqcup \displaystyle{\bigsqcup^{k-1}_{l=0}}X_{n-2} \ar[d] \\ 
\displaystyle{\bigsqcup_{0\leq i<j\leq k-1}}X_{n-2} \ar[r] & \displaystyle{\bigsqcup_{0\leq i<j\leq k}} X_{n-2} } 
\]
is a pushout; 

\bigskip 

\noindent 
(ii) the diagram 
\[
\xymatrix{ 
\displaystyle{\bigsqcup^{k-1}_{l=0}} X_{n-2} \ar[d]_-{s_{k+1}} \ar[r] & \displaystyle{\bigsqcup^{k-1}_{l=0}}X_{n-2}\sqcup X_{n-1} \ar[d] \\ 
\displaystyle{\bigsqcup^{k-1}_{l=0}}X_{n-1}\ar[r] & \displaystyle{\bigsqcup^k_{l=0}}X_{n-1} } 
\]
is a pushout; 

\bigskip 

\noindent 
(iii) the diagram 
\[
\bigsqcup_{0\leq i<j\leq k-1}X_{n-3}\sqcup \bigsqcup^{k-1}_{l=0}X_{n-2} \rightrightarrows 
\bigsqcup^{k-1}_{l=0}X_{n-2}\sqcup X_{n-1} \rightarrow X_{n-1} 
\]
is a coequalizer, where the two displayed maps are defined to be the corresponding maps in the coequalizer defining 
$L_{n-1,k}X$ on the first summand $\bigsqcup_{0\leq i<j\leq k-1}X_{n-3}$, and are defined to be the composites 
\begin{align*} 
& X_{n-2} \xrightarrow{\mathrm{in}_i} \bigsqcup^{k-1}_{l=0}X_{n-2} \to \bigsqcup^{k-1}_{l=0}X_{n-2}\sqcup X_{n-1} \\ 
& X_{n-2} \xrightarrow{s_i} X_{n-1} \to \bigsqcup^{k-1}_{l=0}X_{n-2}\sqcup X_{n-1}
\end{align*} 
on the summand $X_{n-2}$ labelled by $i$ in $\bigsqcup^{k-1}_{i=0}X_{n-2}$ (in this case it is straightforward to check 
that the universal property for a coequalizer is satisfied).  
\end{proof} 

Next, we need a lemma asserting that under certain hypotheses on colimits in $\cC$, 
a canonical square built out of the partial latching objects is a pullback square.  

\begin{lemma} 
\label{lem:degeneracy pullback}
Suppose that $s_k\colon X_n\to X_{n+1}$ is properly extensive for all $n\geq 0$ 
and for all $0\leq k \leq n$.  
Then for every $0\leq k\leq n$ the diagram 
\[
\xymatrix{ 
L_{n,k}X \ar[r] \ar[d] & X_{n} \ar[d]^-{s_k} \\ 
L_{n+1,k}X \ar[r] & X_{n+1} } 
\]
is a pullback.  
\end{lemma} 

\begin{proof} 
Under the hypothesis in the statement of the lemma, we have a coequalizer diagram 
\[
X_n\times_{X_{n+1}}L_{n+1,k}X \longrightarrow 
\bigsqcup_{0\leq i<j\leq k-1} X_{n-1}\times_{X_{n+1}}X_{n-1} \rightrightarrows 
\bigsqcup^{k-1}_{l=0}X_n\times_{X_{n+1}}X_n 
\]
The result then follows from the well-known fact that the diagrams
\[ 
	\xymatrix{ 
		X_{n-1} \ar[r]^-{s_i} \ar[d]_-{s_{j-1}} & X_n\ar[d]^-{s_j} \\ 
		X_n \ar[r]_-{s_i} & X_{n+1} 
	} 
\]
are pullbacks for $i<j$
\end{proof} 
  
We can now give the proof of Proposition~\ref{prop:good_implies_proper}.  

\begin{proof}[Proof of Proposition~\ref{prop:good_implies_proper}] 
We will prove by induction on $n\geq 0$ that the maps $L_{n,k}X\to X_n$ are 
$\bar{h}$-cofibrations for all $0\leq k\leq n$.  The base case is the statement that 
$L_{0,0}X\to X_0$ is an $\bar{h}$-cofibration.  But $L_{0,0}X = \emptyset$ and 
hence the statement is true in this case, since every object of $\cC$ is $\bar{h}$-cofibrant.    

Now we make the inductive assumption that the maps $L_{n-1,k}X\to X_{n-1}$ 
are $\bar{h}$-cofibrations for all $0\leq k\leq n-1$.  We will prove by induction on 
$k$ that $L_{n,k}X\to X_n$ is an $\bar{h}$-cofibration for all $0\leq k\leq n$.  

To start the induction, we again observe that $L_{n,0}X = \emptyset$ and hence $L_{n,0}X\to X_n$ 
is an $\bar{h}$-cofibration.    
Assume then that $L_{n,k}X\to X_n$ is an $\bar{h}$-cofibration for $k\geq 0$ 
and consider the diagram
\begin{equation}
\label{eq:pushout from GJ}
\begin{xy}
	(-12.5,7.5)*+{L_{n-1,k}X}="1"; 
	(-12.5,-7.5)*+{L_{n,k}X}="2"; 
	(12.5,7.5)*+{X_{n-1}}="3"; 
	(12.5,-7.5)*+{L_{n,k+1}X}="4"; 
	(30,-20)*+{X_n}="5"; 
	{\ar "1";"2"}; 
	{\ar "1";"3"}; 
	{\ar "2";"4"}; 
	{\ar "2";"5"}; 
	{\ar^-{s_k} "3";"5"};
	{\ar "4";"5"};
	{\ar "3";"4"};
\end{xy}
\end{equation}
for $0\leq k\leq n-1$. Since the inner square in~\eqref{eq:pushout from GJ} is a pushout, it follows from the assumption that 
$L_{n-1,k}X\to X_{n-1}$ is an $\bar{h}$-cofibration for all $0\leq k\leq n-1$ that 
$L_{n,k}X\to L_{n,k+1}X$ is an $\bar{h}$-cofibration.  By hypothesis, $L_{n,k}X\to X_n$ is 
an $\bar{h}$-cofibration and $s_k\colon X_{n-1}\to X_n$ is an $\bar{h}$-cofibration since 
$X$ is good.  By Lemma~\ref{lem:degeneracy pullback} the outer square in~\eqref{eq:pushout from GJ} is a 
pullback.  Therefore, since $\cC$ satisfies the Lillig condition of Definition~\ref{lillig} 
we conclude that $L_{n,k+1}X\to X_n$ is an $\bar{h}$-cofibration, completing the inductive 
step.  Therefore $L_{n,n}X\to X_n$ is an $\bar{h}$-cofibration, i.e.\ $L_nX\to X_n$ is 
an $\bar{h}$-cofibration, completing the original inductive step.  Hence $X$ is proper.  
\end{proof} 

As an application, we prove the following result, which we need in the proof of
Proposition~\ref{well pointed simp grp} above.

\begin{proposition} 
\label{param lillig}

	Let $\cC = \cKB$.  Then any good simplicial
	object in $\cKB$ is proper. 

\end{proposition} 

\begin{proof}

We deal with condition 2 first. We need to know that the functor $s_n^*\colon
(\cKB)_{/X_{n+1}}\to (\cKB)_{/X_n}$, i.e.\ restriction along the closed inclusion
$s_n\colon X_n\to X_{n+1}$, preserves finite colimits.  In other words, since
$(\cKB)_{/X}\cong \cK_{/X}$ for any object $X$ in $\cKB$, we have to show that
$s_n^*\colon \cK_{/X_{n+1}}\to \cK_{/X_n}$ preserves finite colimits.

A colimit in $\cK_{/X_{n+1}}$ is constructed as a quotient of a coproduct in
$\cK$ and then equipped with the canonical map to $X_{n+1}$.  Therefore it is
sufficient to prove two things: firstly that restriction along  $X_n$
preserves coproducts in $\cK_{/X_{n+1}}$ and secondly that if $q\colon Y\to Z$ is
a quotient map in $\cK_{/X_{n+1}}$ then in the pullback diagram
\[
	\xymatrix{
		X_n\times_{X_{n+1}}Y \ar[d] \ar[r] & Y \ar[d]^-q  \\
		X_n\times_{X_{n+1}}Z \ar[r] & Z
	}
\] 
in $\cK$ the map $X_n\times_{X_{n+1}}Y\to X_n\times_{X_{n+1}}Z$ is a quotient
map.   The first of these things is easy to prove, for the second it is enough
to prove that $X_n\times_{X_{n+1}}Z\to Z$ is a closed  inclusion, since
quotient maps restrict to quotient maps along closed subspaces.  This is clear
however, since  $s_n\colon X_n\to X_{n+1}$ is a closed inclusion, and closed
inclusions pull back along arbitrary maps to closed inclusions.

For the Lillig condition, suppose that
\[ 
	\xymatrix{ 
		A_3 \ar[r] \ar[d] & A_2 \ar[d] \\ 
		A_1 \ar[r] & X 
	} 
\] 
is a pullback diagram in $\cKB$ as in Definition~\ref{lillig} above, i.e.\
the maps $A_1\to X$, $A_2\to X$ and $A_3\to X$ are $\bar{f}$-cofibrations.
From the pushout-product theorem (see \cite{SV}) it follows that
\begin{equation} 
\label{pushprodmap} 
	A_1\cup_{A_3} A_3\otimes I\cup_{A_3} A_2\to X\otimes I 
\end{equation} 
is an $\bar{f}$-cofibration.  This map fits into the commutative diagram 
\[ 
	\xymatrix{ 
		A_1\cup_{A_3} A_3\otimes I\cup_{A_3} A_2 \ar[r] \ar[d] & A_1\cup_{A_3} A_2 \ar[d] \\ 
		X\otimes I \ar[r] & X
	} 
\] 
The pushout of~\eqref{pushprodmap} along $A_1\cup_{A_3} A_3\otimes I\cup_{A_3}
A_2\to A_1\cup_{A_3} A_2$ can be identified with a map
\[ 
	A_1\cup_{A_3} A_2\to X\otimes I\cup_{A_3 \otimes I}A_3
\]
which is also an $\bar{f}$-cofibration.  Therefore, to prove that
$A_1\cup_{A_3} A_2\to X$ is an $\bar{f}$-cofibration it suffices to prove that
$A_1\cup_{A_3} A_2\to X$ is a retract of $A_1\cup_{A_3} A_2\to X\otimes
I\cup_{A_3\otimes I}A_3$.   Suppose $(u_1,h_1)$ and $(u_2,h_2)$ are
representations of $(X,A_1)$ and $(X,A_2)$ as fiberwise NDR pairs.  As in
\cite{Lillig} define a map $u\colon X\to X\otimes I\cup_{A_3\otimes I}A_3$ by
\[ 
	u(x) = \begin{cases} 
	[x,u_1(x)/(u_1(x) + u_2(x))] & \text{if}\ x\notin A_3, \\ 
	[x,0] & \text{if}\ x\in A_3.  
	\end{cases} 
\] 
Then it is easy to check that $u(x) = [x,0]$ if $x\in A_1$ and $u(x) = [x,1]$
if $x\in A_2$.  This map exhibits $A_1\cup_{A_3} A_2\to X$ as a retract, as
required.   
\end{proof}
    
We do not know if the Lillig condition holds more generally; the
proof we have given (which is a re-working of Lillig's original proof) uses
crucially the characterization of  $\bar{f}$-cofibrations in terms of
fiberwise NDR pairs.  We note that the result is false in general if
$\bar{f}$-cofibrations are replaced by $f$-cofibrations.

\end{appendices}

\end{document}